\documentclass{article}

\usepackage[margin=1.5in]{geometry}

\usepackage[nodisplayskipstretch]{setspace}

\makeatletter
\usepackage{bm}
\g@addto@macro\bfseries{\boldmath}
\usepackage[english]{babel}

\usepackage[T1]{fontenc}
\RequirePackage{amsmath}
\RequirePackage{amssymb,amsfonts}
\RequirePackage{amsthm}
\RequirePackage{amsrefs}
\RequirePackage{xcolor} 
\usepackage{tikz}
\usepackage{tikz-cd}
\usetikzlibrary{decorations.pathreplacing,calligraphy,arrows}
\RequirePackage[pdfusetitle]{hyperref} 
\hypersetup{
	colorlinks,
	linkcolor={blue!80!black},
	citecolor={blue!50!black},
	urlcolor={blue!80!black},
	bookmarksnumbered,
	bookmarksdepth=2,
	pdfencoding=unicode|auto,
	pdftitle={\@title},
	pdfauthor={\@author},
	pdfencoding=unicode,
}

\RequirePackage{thmtools}
\makeatletter
\def\thesubsection{\thesection\ifnum\c@subsection=0\relax\else.\fi\Alph{subsection}}
\makeatother
\declaretheoremstyle[
headfont=\normalfont\bfseries\textup,
notefont=\normalfont\bfseries, notebraces={(}{)},
bodyfont=\normalfont\itshape,
headformat={\textup{\NUMBER.\ \NAME\NOTE}},
]{JHstyle}

\declaretheorem[style=JHstyle,
numberwithin=subsection,
name=Theorem,
refname={theorem,theorems},
Refname={Theorem,Theorems}%
]{theorem} 

\declaretheorem[style=JHstyle,
sibling=theorem,
name=Definition,
refname={definition,definitions},
Refname={Definition,Definitions}%
]{definition}
\declaretheorem[style=JHstyle,
sibling=theorem,
name=Result,
refname={result,results},
Refname={Result,Results}%
]{result}
\declaretheorem[style=JHstyle,
sibling=theorem,
name=Lemma,
refname={lemma,lemmas},
Refname={Lemma,Lemmas}%
]{lemma}
\declaretheorem[style=JHstyle,
sibling=theorem,
name=Corollary,
refname={corollary,corollaries},
Refname={Corollary,Corollaries}]{corollary}
\declaretheorem[name=Claim,refname={claim,claims},Refname={Claim,Claims}]{claim}
\makeatletter
\@addtoreset{claim}{theorem}
\makeatother

\providecommand{\hlink}[1]{\Autoref{#1}}

\makeatother

\usepackage{graphicx}

\newcommand{\ffrak}[1]{\mathrm{#1}}

\newcommand{\axm}[1]{\textup{\sffamily #1}}
\providecommand{\bb}[1]{\mathbb{#1}}
\providecommand{\bbb}[1]{\mathbb{#1}}
\declaretheorem[style=JHstyle,sibling=theorem,name=Questions,refname={questions,questions},Refname={Questions,Questions}]{questions}

\newcommand{\nothing}[1]{}

\newcommand{\I}{\textup{\textbf{I}}}
\newcommand{\II}{\textup{\textbf{I\hspace{-.7pt}I}}}
\newcommand{\et}{&&}
\newcommand{\playerI}{\text{\I:}&}
\newcommand{\playerII}{\text{\II:}\et}

\newenvironment{game*}[1][]{
	\[{#1}\qquad\begin{array}{r *{20}{c}}
	}{
	\end{array}\]\ignorespacesafterend
}

          \DeclareMathOperator{\Con}{Con} \DeclareMathOperator{\rank}{rank}  \DeclareMathOperator{\cof}{cof}   \DeclareMathOperator{\cp}{cp} \DeclareMathOperator{\lh}{lh} \DeclareMathOperator{\str}{str}  \DeclareMathOperator{\Ult}{Ult}

\title{Weak Indestructibility and Reflection}
\author{James Holland}
\date{2022-04-12}

\newcommand{\bbone}{\textbf{1}}

\usepackage{scalerel}
\DeclareMathOperator*{\bigast}{{\smash{\scalebox{2.5}{\raisebox{-.9ex}{$*$}}}}}

\newcommand{\dblq}[1]{\text{``}#1\text{''}}

\begin{document}
	
	\maketitle

	\begin{abstract}
		We establish an equiconsistency between (1) weak indestructibility for all \(\kappa +2\)-degrees of strength for cardinals \(\kappa \) in the presence of a proper class of strong cardinals, and (2) a proper class of cardinals that are strong reflecting strongs.  We in fact get weak indestructibility for degrees of strength far beyond \(\kappa +2\), well beyond the next inaccessible limit of measurables (of the ground model).  One direction is proven using forcing and the other using core model techniques from inner model theory.  Additionally, connections between weak indestructibility and the reflection properties associated with Woodin cardinals are discussed.  This work is a part of my upcoming thesis \cite{thesis}.
	\end{abstract}

	\tableofcontents

	\newpage
	
	\section{Introduction}
	
	Theorems like compactness for first-order logic tell us there is a great degree of ambiguity for mathematical and set theoretic concepts.  But much of this disappears if we restrict ourselves to only considering models that, in some sense, properly interpret membership (i.e.\ that are transitive).  But this does not get rid of all ambiguity.  The advent of \emph{forcing} as a method of set theory has revealed that even with this restriction, there's a great amount of ambiguity not about the membership relation, but instead about what sets exist.
	
	One response to this is to come up with axioms that help sharpen our conception of the universe so as to become immune to methods of forcing in some sense, and say more definitively what exists.  For example, the axiom of \dblq{\(\mathrm{V}=\mathrm{L}\)} yields models that are immune to forcing.  Under certain assumptions, the core model \(\ffrak{K}\) has this property too.  Forcing axioms can be considered in a similar way, essentially stating that we've already forced as much as we can.
	
	The consistency of most of these axioms is unfortunately not something we can know, as they often carry with them \emph{large cardinal} hypotheses.  Such hypotheses are used as a standard measure of the strength of statements: if we want to know how strong a statement is, we show it is equiconsistent with some large cardinal axiom.  The question then becomes to what extent can these large cardinals be immune to forcing?
	
	Generally speaking, we cannot ensure large cardinal properties are immune to forcing—\emph{indestructible}—because we may simply add via forcing a bijection between the cardinal \(\kappa \) and \(\aleph _0\).  But if we restrict our forcing to be, say, smaller than \(\kappa \), we can get preservation of some properties, especially those involving elementary embeddings since measures in the ground model generate measures in the generic extension.  Large forcings can still pose a problem, but the answer isn't as clear if we restrict our attention to large posets that don't affect anything \dblq{small}.  These are the notions we will investigate here, and mostly we will consider \(<\kappa \)-strategi\-cally closed, \(\le\kappa \)-distributive posets.
	
	This topic starts with Laver's preparation for making a supercompact cardinal \(\kappa \) indestructible (by \(<\kappa \)-directed closed forcings) in \cite{laverprep}.  Since then, there has been a great deal of literature about the limits of this for other cardinals and varying degrees of strength.\footnote{By a \emph{degree of strength} of a caridnal \(\kappa \), I mean an ordinal \(\rho \) such that \(\kappa \) is \(\rho \)-strong, defined with \hlink{strongdef}.}  Ignoring for the moment what exactly a lot of this terminology means, a short selection of results in this area is the following.
	\begin{itemize}
		\item \cite{apterhamkins} explores making all degrees of supercompactness and strength indestructible while there is a supercompact.
		\item \cite{hamkinssuper} explores the ways in which indestructible cardinals can be made destructible and subsequently resurrected.
		\item \cite{aptersargsyan} explores a weaker version of indestructibility for strength while there is a strong cardinal, and gets an equiconsistency result: universal weak indestructibility (while there is a strong cardinal) is equiconsistent with a \emph{hyperstrong} cardinal.
		\item \cite{apterapplications} further explores this weaker version of indestructibility for strength to get (weak) indestructibility for lots of (strong) supercompact cardinals, and again establishes an equiconsistency for this indestructibility for many strongs from a hyperstrong cardinal (a proper class if we cut off the universe at the hyperstrong).
	\end{itemize}
	This work further explores \cite{aptersargsyan} and \cite{apterapplications}.  Generally speaking, these works focus on exploring indestructibility for large degrees of strength and supercompactness in the presence of many strong or supercompact cardinals.  Ostensibly, this is a harder task than showing weak indestructibility for smaller degrees of strength, simply because a strong cardinal's degrees of strength are arbitrarily large.  But instead, this work shows that to get universal weak indestructibility for small degrees of strength and supercompactness, we actually need a large \emph{increase} in consistency strength, much more than a proper class of hyperstrong or hypercompact cardinals, just one of which is used in \cite{aptersargsyan} and \cite{apterapplications} to get (many) weakly indestructible strongs and supercompacts.
	
	There is a balance to the amount of indestructibility one can have, and generally speaking, indestructibility for large degrees of strength will conflict with indestructibility for small degrees of strength. The largest sense of \dblq{small} degrees is degrees of strength below the next cardinal \(\kappa \) that is \(\kappa +2\)-strong, and the smallest sense of \dblq{large} degrees is degrees slightly above such a cardinal.  So the universal indestructibility results of \cite{aptersargsyan} and \cite{apterhamkins} critically proceed by making sure there is no measurable above the resulting strong or supercompact cardinal, thus being incompatible with the existence of multiple strong or supercompact cardinals. The work in \cite{apterapplications} considers multiple strong and supercompact cardinals with indestructibility, but as a consequence, ignores indestructibility for smaller degrees of strength.
	
	The following theorem was the main goal of \cite{aptersargsyan}.
	\begin{theorem}[Apter\textendash Sargsyan]\label{aptsargmain}
		The following are equiconsistent with \(\axm{ZFC}\):
		\begin{enumerate}
			\item There is a hyperstrong cardinal.
			\item There is a strong cardinal and universal weak indestructibility for all degrees of strength holds.
		\end{enumerate}
	\end{theorem}
	Hence a hyperstrong cardinal is able to yield indestructibility for all small degrees of strength, and produce a strong cardinal (with also indestructible strength).  Considering the balance of indestructible degrees of strength, a natural strengthening of (2) is having a proper class of strong cardinals while still ensuring universal weak indestructibility for small degrees of strength.  A natural guess at the consistency strength of this is the existence of many hyperstrong cardinals.  But this is insufficient, and indeed the consistency strength of this is strictly larger than a proper class of hyperstrong cardinals.  The following is one of the main results of the document that ensures this increase in strength, assuming for the moment that a single strong reflecting strongs cardinal (with a strong above it) gives the consistency of a proper class of hyperstrongs.
	
	\begin{theorem}[Main Goal]\label{maingoal}
		The following are equiconsistent with \(\axm{ZFC}\):
		\begin{enumerate}
			\item There is a proper class of strongs reflecting strongs.
			\item There is a proper class of strong cardinals and weak indestructibility for any cardinal \(\kappa \)'s \(\kappa +2\)-strength.
			\item There is a proper class of strong cardinals and weak indestructibility for any cardinal \(\kappa \)'s \(\lambda \)-strength where \(\lambda \) is below the next measurable limit of measurables above \(\kappa \).
		\end{enumerate}
	\end{theorem}
	
	In my thesis~\cite{thesis}, we also establish a similar result for supercompactness, and the following just by continuing the preparation from \hlink{maingoal} through to a Woodin cardinal, and a separate preparation where we ignore small degrees of strength.
	
	\begin{corollary}[Side Result]\label{woodingoal}
		The following are equiconsistent with \(\axm{ZFC}\):
		\begin{enumerate}
			\item There is a Woodin cardinal.
			\item There is a Woodin cardinal \(\delta \) such that weak indestructibility for any \(\kappa \)'s \(\kappa +2\)-strength holds below \(\delta \).
			\item There is a Woodin cardinal \(\delta \) such that every \(<\delta \)-strong cardinal has weakly indestructible \(<\delta \)-strength.
		\end{enumerate}
	\end{corollary}
	
	This gives the consistency of a large number of weakly indestructible strong cardinals from a Woodin cardinal.  Moreover, it tells us that weak indestructibility for large degrees of strength and for small degrees of strength below a Woodin cardinal are equiconsistent.  So while weak indestructibility for small degrees (with, say, two strong cardinals) is strictly stronger than for a proper class of weakly indestructible strong cardinals, this difference levels out as we approach a Woodin.
	
	\subsection{Background}
	
	Throughout we will use the following notation.
	
	\begin{definition}\label{notationconventions}\hfill
		\begin{itemize}
			\item For \(\delta \in \mathrm{Ord}\), write \(\delta ^{+\sharp}\) for the least measurable cardinal \(\kappa >\delta \).
			\item Write \(\delta ^{+\sharp\sharp}\) for the least cardinal \(\kappa >\delta \) that is \(\kappa +2\)-strong.
			\item Write \(\delta ^{+\P}\) for the least strong cardinal \(\kappa >\delta \).
			\item For a transitive model of set theory, \(\ffrak{M}\), and \(\lambda \in \mathrm{Ord}^{\ffrak{M}}\), we write \(\ffrak{M}\mid\lambda \) for \(\ffrak{V}_\lambda ^{\ffrak{M}}\), the \(\lambda \)-th stage of the cumulative hierarchy in \(\ffrak{M}\).
			\item For a poset \(\bbb{R}\in \mathrm{V}\), we write \(\ffrak{V}^{\bbb{R}}\) for an arbitrary generic extension of \(\ffrak{V}\) by \(\bbb{R}\), occasionally working below some particular condition in \(\bbb{R}\).
		\end{itemize}
	\end{definition}
	
	\begin{definition}\label{fourindestructdef}
		Let \(\kappa \) be a cardinal, \(\varphi \) a property of cardinals, and \(\psi \) a property of posets.
		\begin{itemize}
			\item \(\kappa \) has \(\varphi \) as \emph{indestructible} by posets with \(\psi \) iff
			\[\forall \bbb{P}\text{ a poset}\ (\psi (\bbb{P})\rightarrow \bbb{P}\Vdash \dblq{\varphi (\check\kappa )})\text{.}\]
			\item \(\kappa \) has \(\varphi \) as \emph{weakly indestructibile} by posets with \(\psi \) iff it's indestructible by \(\le\kappa \)-distributive posets with \(\psi \).
		\end{itemize}
	\end{definition}
	So the usual Laver indestructibility refers to indestructibility for all degrees of supercompactness of a single (supercompact) cardinal \(\kappa \) by \(<\kappa \)-directed closed posets.  For strength, we usually consider weak indestructibility by \(<\kappa \)-strategi\-cally closed posets, and thus indestructibility \(<\kappa \)-strategi\-cally closed, \(\le\kappa \)-distributive posets.  \(<\kappa \)-strategic closure is a significant weakening of \(<\kappa \)-directed closure and is defined as follows.

	\begin{definition}\label{stratcloseddef}
		Let \(\bbb{P}\) be a poset and \(\kappa ,\lambda \) ordinals.  The game \(\mathcal{G}_{\bbb{P}}^\lambda \) is the two person game of length \(\le\lambda \)
		\begin{game*}
			\playerI p_0=\bbone^{\bbb{P}}\et p_2\et \cdots \et p_\omega \et p_{\omega +2}\et\cdots \\
			\playerII p_1\et p_3\et\cdots \et p_{\omega +1}\et\cdots \text{,}
		\end{game*}
		where \(p_\alpha \le^{\bbb{P}} p_\beta \) for \(\alpha \ge\beta \), and assuming such a \(p_\alpha \) exists, \(\textup{\textbf{I}}\) plays \(p_\alpha \) for even \(\alpha < \lambda \) (including limit \(\alpha \)) and \(\textup{\textbf{II}}\) plays \(p_\alpha \in \bb{P}\) for odd \(\alpha <\lambda \).  \(\textup{\textbf{I}}\) wins iff players play on each turn: the resulting sequence of \(p_\alpha \)s has length \(\lambda \).
		\begin{itemize}
			\item \(\bb{P}\) is \emph{\(\le \kappa \)-strategi\-cally closed} iff \(\textup{\textbf{I}}\) has a winning strategy in \(\mathcal{G}_{\bb{P}}^{\kappa +1}\).
			\item \(\bb{P}\) is \emph{\(\kappa \)-strategi\-cally closed} iff \(\textup{\textbf{I}}\) has a winning strategy in \(\mathcal{G}_{\bb{P}}^\kappa \).
			\item \(\bb{P}\) is \emph{\(<\kappa \)-strategi\-cally closed} iff \(\bb{P}\) is \(\le\alpha \)-strategi\-cally closed for all ordinals \(\alpha <\kappa \).
		\end{itemize}
	\end{definition}
	Basically whereas \(<\kappa \)-closure ensures we can always extend a \(\leqslant^{\bbb{P}}\)-decreasing sequence \(\langle p_\alpha :\alpha <\lambda \rangle \) whenever \(\lambda <\kappa \), \(\kappa \)-strategic closure only gives control over half of the sequence, relying on \(\textup{\textbf{I}}\)'s strategy to extend at limits.  
	
	Despite their differences, the similarities between \(<\kappa \)-closed and \(\kappa \)-strategi\-cally closed preorders are enough to show to allow some arguments about \(<\kappa \)-closure to go through about \(\kappa \)-strategic closure.  For example, \(\kappa \)-strategi\-cally closed preorders are \(\le \kappa \)-distributive by basically the same proof as with full closure. There are many other common results in iterated forcing where using strategic closure vs.\ full closure makes no difference~\cite{cummings}.
	
	\begin{definition}\label{κdistribdef}
		Let \(\kappa \) be a cardinal. A preorder \(\bbb{P}\) is \emph{\(<\kappa \)-distributive} iff for every collection \(\mathcal{D}\) of open, dense sets of \(\bbb{P}\), if \(|\mathcal{D}|<\kappa \) then \(\bigcap \mathcal{D}\neq \emptyset \) is open, dense.  We may similarly define \(\le\kappa \)-distributive.
	\end{definition}
	
	\begin{corollary}\label{stratimpliesdistrib1}
		For any infinite cardinal \(\kappa \), if \(\bbb{P}\) is \(<\kappa \)-strategi\-cally closed, then \(\bbb{P}\) is \(<\kappa \)-distributive.
	\end{corollary}
	
	It's not hard to see that if \(\kappa \) is measurable, its \(\kappa +1\)-strength is weakly indestructible.  So easy models of universal indestructibility for small degrees of strength are those with no large cardinals, or those with only measurables.  Obviously if a poset \emph{isn't} \(<\kappa \)-distributive, it can collapse the strength of \(\kappa \), e.g. through a trivial collapse \(\ffrak{Col}(\omega ,\kappa )\).
	
	\begin{corollary}\label{measureweakindes}
		Let \(\kappa \) be measurable.  Therefore \(\kappa \)'s measurability (i.e.\ its \(\kappa +1\)-strength) is weakly indestructible.
	\end{corollary}
	\begin{proof}
		If \(U\in \mathrm{V}\) is a measure and \(\bbb{P}\in \mathrm{V}\) is a \(\le\kappa \)-distributive poset, then \(U\in \mathrm{V}\subseteq \mathrm{V}^{\bbb{P}}\) is still a measure in \(\ffrak{V}^{\bbb{P}}\).\qedhere
	\end{proof}
	
	The rest of this document will assume familiarity with strong cardinals and their fundamental properties.
	
	\begin{definition}\label{strongdef}
		\hfill
		\begin{itemize}
			\item A cardinal \(\kappa \) is \emph{\(\lambda \)-strong} iff there is an elementary embedding \(j:\mathrm{V}\rightarrow \mathrm{M}\) (as a class of \(\ffrak{V}\)) with \(\cp(j)=\kappa \), \(j(\kappa )>\lambda \), and \(\ffrak{V}\mid\lambda =\ffrak{M}\mid\lambda \).
			\item A cardinal is \emph{strong} iff it is \(\lambda \)-strong for every \(\lambda \in \mathrm{Ord}\).
			\item We call embeddings (extenders) witnessing \(\lambda \)-strength \emph{\(\lambda \)-strong embeddings (extenders)}.
		\end{itemize}
	\end{definition}
	
	Beyond their definition, strong cardinals are useful for the following reflection principle.
	
	\begin{lemma}[\(\Sigma _2\)-reflection]\label{Σ2ref}
		Let \(\kappa \) be a strong cardinal.  Suppose \(\ffrak{V}\mid \delta \vDash \dblq{\varphi (\vec{x})}\) for some formula \(\varphi \) and parameters \(\vec{x}\in \mathrm{V}\mid \kappa \).  Therefore there are unboundedly many \(\alpha <\kappa \) such that \(\ffrak{V}\mid \alpha \vDash \dblq{\varphi (\vec{x})}\).
	\end{lemma}
	\begin{proof}
		Let \(j:\mathrm{V}\rightarrow \mathrm{M}\) be a \(\delta \)-strong embedding so that \(j(\vec{x})=\vec{x}\), \(\delta <j(\kappa )\), and \(\ffrak{M}\mid \delta =\ffrak{V}\mid \delta \vDash \dblq{\varphi (\vec{x})}\).  In \(\ffrak{M}\), there is then some level of the cumulative hierarchy below \(j(\kappa )\) that satisfies \(\varphi (j(\vec{x}))\) and so by elementarity, in \(\ffrak{V}\) there's a level \(\mathrm{V}\mid \alpha \) of the cumulative hierarchy with \(\rank(\vec{x})<\alpha <\kappa \) satisfying \(\varphi (\vec{x})\).  Because we could have thrown in any fixed ordinal \(\beta <\kappa \) as a useless parameter into \(\vec{x}\), the idea above gives an \(\alpha >\beta \) with this property. \qedhere
	\end{proof}
	
	We now establish some definitions related to \cite{aptersargsyan}, explaining \hlink{aptsargmain}.  
	\begin{definition}\label{hyperstrongdef}
		Let \(\kappa \) and \(\alpha >0\) be ordinals.  We define by transfinite recursion what it means for \(\kappa \) to be \emph{\(\alpha \)-hyperstrong}.
		\begin{itemize}
			\item \(\kappa \) is \emph{\(0\)-hyperstrong} iff \(\kappa \) is strong.
			\item \(\kappa \) is \emph{\(<\alpha \)-hyperstrong} iff \(\kappa \) is \(\xi \)-hyperstrong for every \(\xi <\alpha \).
			\item \(\kappa \) is \emph{\(\alpha \)-hyperstrong} iff for any \(\lambda >\kappa \), there is an extender \(E\) giving an ultrapower embedding \(j_E:\mathrm{V}\rightarrow \mathrm{M}\) with \(\cp(j_E)=\kappa \), \(j_E(\kappa )>|\mathrm{V}\mid \lambda |\), \(\mathrm{V}\mid \lambda \subseteq \mathrm{M}\), and \(\ffrak{M}\vDash \dblq{\kappa \text{ is }<\alpha \text{-hyperstrong}}\).
			\item \(\kappa \) is \emph{hyperstrong} iff \(\kappa \) is \(\xi \)-hyperstrong for every ordinal \(\xi \).
		\end{itemize}
	\end{definition}
	We will work with hyperstrongs mostly to show that they are insufficient in establishing the main result of \hlink{maingoal}.  We will also make use of strongs reflecting strongs. We will discuss the interaction between these two definitions in the next subsection.
	
	\begin{definition}\label{srsdef}
		Let \(\kappa \) be a cardinal and \(\lambda \) an ordinal.
		\begin{itemize}
			\item We say \(\kappa \) is \emph{\(\lambda \)-strong reflecting strongs (\(\lambda \)-\emph{srs})} iff there's a \(\lambda \)-strong \((\kappa ,\lambda )\)-extender \(E\) giving an elementary \(j_E:\mathrm{V}\rightarrow \mathrm{M}=\Ult_E(\mathrm{V})\) with
			\[\{\xi <\lambda :\xi \text{ is strong}\}=\{\xi <\lambda :\ffrak{M}\vDash \dblq{\xi \text{ is strong}}\}\text{.}\]
			\item We call such an embedding (extender) a \emph{\(\lambda \)-srs embedding (\emph{extender})}.
			\item We say \(\kappa \) is \emph{srs} iff \(\kappa \) is \(\lambda \)-srs for every \(\lambda >\kappa \); and we say \(\kappa \) is \(<\lambda \)-srs if \(\kappa \) is \(\alpha \)-srs for each \(\alpha <\lambda \).
		\end{itemize}
	\end{definition}
	
	For the most part, we only care about strongs reflecting strongs when they have strongs above them since if there is no strong above \(\kappa \), \(\kappa \) is srs iff \(\kappa \) is \(1\)-hyperstrong.  But under the not-uncommon assumption of a strong above, this is a strengthening of hyperstrong cardinals both in the sense of being hyperstrong and in having a (much) higher consistency strength. The  (2) implies (1) direction of \hlink{maingoal} then tells us that a proper class of hyperstrong cardinals is insufficient to ensure universal weak indestructibility for very small degrees of strength with a proper class of strongs.
	
	For now, we end this subsection with the following definition.
	
	\begin{definition}\label{UWISSdef}
		\emph{\(\axm{UWISS}\)} \emph{(Universal weak indestructibility for small degrees of strength)} is the proposition that for all \(\kappa \), if \(\kappa \) is \(\kappa +2\)-strong, then \(\kappa \)'s \(\kappa +2\)-strength is weakly indestructible by \(<\kappa \)-strategi\-cally closed posets.
	\end{definition}
	This will be equiconsistent with every \(\rho \)-strong \(\kappa \) having weakly indestructible \(\rho \)-strength for \(\rho \) below the next measurable limit of measurables above \(\kappa \).

	\subsection{Limiting results}

	The limit imposed on universal indestructibility was explored in \cite{apterhamkins} for full indestructibility, but we may consider weak indestructibility in a similar way.  The main limiting result is the following with the proof adapted from \cite{apterhamkins} and the Superdestruction Theorem III from \cite{hamkinssuper}.
	
	\begin{result}\label{hamkinsoopsy}
		Assume \(\axm{GCH}\).\footnotemark{} Suppose \(\kappa \) is strong and this is not destroyed by \(\ffrak{Add}(\kappa ^+,1)\).  Suppose \(\kappa ^{+\sharp\sharp}\) exists. Therefore, there are arbitrarily large \(\delta <\kappa \) whose \(\delta +2\)-strength is destroyed by \(\ffrak{Add}(\delta ^+,1)\).
	\end{result}
	\footnotetext{Note that the assumption of \(\axm{GCH}\) here isn't much of a problem, since in the context of indestructibility by a large class of posets, one can just force the appropriate instances with Cohen forcing.  More precisely, for the theorem to go through, we merely need \(\lambda =\kappa ^{+\sharp\sharp}\) to have \(2^{2^\lambda }=\lambda ^{++}\) or else we must consider a \(\lambda >\kappa \) that is \(\gamma \)-strong where \(2^{2^{\lambda }}=\lambda ^{+\gamma }\).}
	\begin{proof}
		Let \(A\) be \(\ffrak{Add}(\kappa ^+,1)\)-generic over \(\ffrak{V}\). By hypothesis, \(\kappa \) is still strong in \(\ffrak{V}[A]\) and so we get a sufficiently strong embedding \(j:\mathrm{V}[A]\rightarrow \mathrm{M}[j(A)]\) with \(\cp(j)=\kappa \) so that \(\lambda =\kappa ^{+\sharp\sharp}\) is still \(\lambda +2\)-strong in \(\ffrak{M}[j(A)]\).
		\begin{claim}\label{hamkinsoopsycl1}
			\(\lambda \)'s \(\lambda +2\)-strength is destroyed by \(\ffrak{Add}(\lambda ^+,1)\) in \(\ffrak{V}[A]\).
		\end{claim}
		\begin{proof}
			Suppose not.  Let \(G\) be \(\ffrak{Add}(\lambda ^+,1)\)-generic over \(\ffrak{V}[A]\) and let \(i:\mathrm{V}[A*G]\rightarrow \mathrm{N}[i(A*G)]\) be \(\lambda +2\)-strong with \(\cp(i)=\lambda \).  Since \(A\subseteq \kappa <\cp(i)\), \(i(A*G)=A*i(G)\).  Without loss of generality, \(i\) is generated by extenders such that \(\mathrm{N}[A*i(G)]\) is closed under \(\lambda \)-sequences of \(\mathrm{V}[A*G]\).
			
			By gap forcing \cite{GapForcing}, \(i\upharpoonright\mathrm{V}:\mathrm{V}\rightarrow \mathrm{N}\) is a class of \(\mathrm{V}\), is closed under \(\lambda \)-sequences, and is still \(\lambda +2\)-strong.  It follows that \(G\subseteq \lambda ^+\) is in \(\mathrm{H}_{\lambda ^{++}}^{\ffrak{V}[A*G]}=\mathrm{H}_{\lambda ^{++}}^{\ffrak{N}[A*i(G)]}\) and hence in \(\mathrm{N}[A*i(G)]\), but wasn't added by \(i(\ffrak{Add}(\lambda ^+,1))\) by \(\le i(\lambda ^+)\)-distributivity in \(\ffrak{N}[A]\).  Thus \(G=\dot G_A\in \mathrm{N}[A]\) for some \(\ffrak{Add}(\kappa ^+,1)\)-name \(\dot G\) which can be thought of as a function from \((\lambda ^+)^{\ffrak{V}}=(\lambda ^+)^{\ffrak{N}}\) to antichains of \(\ffrak{Add}(\kappa ^+,1)\), and is therefore in the hereditarily \(<(\lambda ^{++})^{\ffrak{N}}=(\lambda ^{++})^{\ffrak{V}}\)-sized sets \(\mathrm{H}_{\lambda ^{++}}^{\ffrak{N}}=\mathrm{H}_{\lambda ^{++}}^{\ffrak{V}}\).  The ability to consider all of this in \(\ffrak{V}\) means \(G\in \mathrm{V}[A]\), a contradiction.\qedhere
		\end{proof}
		Without loss of generality, \(\ffrak{M}[j(A)]\) has enough agreement with \(\ffrak{V}[A]\) to witness this fact as well.  Since \(\kappa <\lambda <j(\kappa )\), this is reflected down into \(\ffrak{V}[A]\): there are arbitrarily large \(\delta <\kappa \) such that \(\delta \)'s \(\delta +2\)-strength is destroyed by \(\ffrak{Add}(\delta ^+,1)\) in \(\ffrak{V}[A]\).  It follows that \(\ffrak{Add}(\kappa ^+,1)*\ffrak{Add}(\delta ^+,1)\) is \(<\delta \)-directed closed and \(\le\delta \)-distributive in \(\ffrak{V}\) but destroys \(\delta \)'s \(\delta +2\)-strength.  Since \(\ffrak{Add}(\kappa ^+,1)\) is sufficiently distributive, it can't add to nor destroy \(\delta \)'s \(\delta +2\)-strength.  So it must be that \(\ffrak{Add}(\delta ^+,1)\) destroyed \(\delta \)'s \(\delta +2\)-strong\qedhere
	\end{proof}
	
	And this result generalizes to larger degrees of strength with the same proof.
	
	\begin{corollary}\label{hamkinsoopsygen}
		Assume \(\axm{GCH}\).  Let \(\mu \in \mathrm{Ord}\).  Suppose \(\kappa \) is strong, and this is not destroyed by \(\ffrak{Add}(\kappa ^{+\mu },1)\).  Suppose there is a \(\lambda >\kappa ^{+\mu +1}\) that is \(\lambda +\mu +1\)-strong.  Therefore
		\begin{enumerate}
			\item There are unboundedly many \(\delta <\kappa \) such that \(\delta \) is \(\delta +\rho \)-strong for some \(\rho \), but this is destructible by \(\ffrak{Add}(\delta ^{+\rho },1)\).
			\item If \(\mu <\kappa \), then there are arbitrarily large \(\delta <\kappa \) such that \(\delta \)'s \(\delta +\mu +1\)-strength is destructible by \(\ffrak{Add}(\delta ^{+\mu },1)\).
		\end{enumerate}
	\end{corollary}
	In particular, one cannot have universal weak indestructibility for \emph{all} degrees of strength with two strong cardinals. This also easily generalizes to supercompactness in place of strength.
	
	The contrapositive to \hlink{hamkinsoopsy} details one aspect of what goes wrong if one naïvely tries to continue the preparation from either \cite{apterhamkins} or \cite{aptersargsyan}.
	\begin{corollary}\label{hamkinsoopsyv2}
		Assume \(\axm{GCH}\). Suppose \(\kappa \) is strong and \(\kappa ^{+\sharp\sharp}\) exists.  Suppose every \(\delta <\kappa \) that is \(\delta +2\)-strong has this indestructible by \(\ffrak{Add}(\delta ^+,1)\).  Therefore \(\kappa \)'s strength is destroyed by \(\ffrak{Add}(\kappa ^+,1)\).
	\end{corollary}
	All of this is just to justify that one simultaneously cannot have \(\axm{UWISS}\) with any reasonable amount of indestructibility for any large degrees of strength.  Moreover, \hlink{hamkinsoopsy} and \hlink{hamkinsoopsyv2} aren't unique to weak indestructibility as noted in the original Theorems 10, 11, and 12 of \cite{apterhamkins}, which hold for full indestructibility.  Really what \hlink{hamkinsoopsyv2} and \hlink{hamkinsoopsy} tell us is that the only way to ensure universal indestructibility for \emph{all} degrees of strength with a strong cardinal \(\kappa \) is to ensure \(\kappa ^{+\sharp\sharp}\) doesn't exist.  And this is exactly what \cite{apterhamkins} and \cite{aptersargsyan} do.
	
	Let us now show why—assuming the equiconsistency of \hlink{maingoal}—hyper\-strongs are insufficient to establish \(\axm{UWISS}\) with a proper class of (weakly destructible) strongs.  To do this, we need to unfortunately look at the fine details of calculating hyperstrongs and strongs reflecting strongs inside the cumulative hierarchy.  A straightforward, somewhat tedious induction gives the following, recalling \hlink{hyperstrongdef} for the definition of a hyperstrong cardinal.
	
	\begin{corollary}\label{strhypstr}
		Fix \(\kappa ,\alpha \in \mathrm{Ord}\) with \(\kappa \) infinite. Therefore the following are equivalent.
		\begin{enumerate}
			\item \(\kappa \) is \(<\alpha \)-hyperstrong.
			\item for all limit \(\delta \ge|\kappa +\alpha |^+\), \(\ffrak{V}\mid\delta \vDash \dblq{\kappa \text{ is }<\alpha \text{-hyperstrong}}\).
		\end{enumerate}
	\end{corollary}
	\begin{lemma}\label{srsishypercl}
		Let \(\kappa \) be strong such that \(\kappa ^{+\P}\) exists.  Therefore,
		\begin{enumerate}
			\item If \(\kappa \) is \(<\kappa ^{+\P}\)-hyperstrong, then \(\kappa \) is hyperstrong.
			\item If \(\ffrak{V}\mid {\kappa ^{+\P}}\vDash \dblq{\kappa \text{ is hyperstrong}}\), then \(\kappa \) is \(<\kappa ^{+\P}\)-hyperstrong and hence hyperstrong by (1).
		\end{enumerate}
	\end{lemma}
	\begin{proof}
		\begin{enumerate}
			\item Suppose \(\kappa \) is not hyperstrong.  By reflection, \[\ffrak{V}\mid \delta \vDash \dblq{\exists \gamma \ (\kappa \text{ is not }\gamma \text{-hyperstrong})\wedge \text{there is no largest cardinal}}\tag{$*$}\]
			for some \(\delta >\alpha \).  By \hlink{Σ2ref}, we get arbitrarily large \(\delta <{\kappa ^{+\P}}\) such that ($*$) holds.  Since there is no largest cardinal in \(\ffrak{V}\mid \delta \), such \(\delta \) satisfy \(\delta >|\kappa +\gamma |^+\) for any \(\gamma \) as in (\(*\)).  \hlink{strhypstr} therefore implies \(\kappa \) is not \(\gamma \)-hyperstrong for some \(\gamma <{\kappa ^{+\P}}\), a contradiction.
			\item Suppose \(\kappa \) is not \(\alpha \)-hyperstrong for some \(\alpha <{\kappa ^{+\P}}\).  This fact is reflected to some level \(\ffrak{V}\mid \delta \), and by \hlink{Σ2ref}, to arbitrarily large initial segments \(\ffrak{V}\mid \delta \) for \(\delta <{\kappa ^{+\P}}\).  In particular, we get some \(\delta _0\) with \(|\kappa +\alpha |^+<\delta _0<{\kappa ^{+\P}}\) such that \(\ffrak{V}\mid {\delta _0}\vDash \) \dblq{\(\kappa \) is not \(\alpha \)-hyperstrong}.  But applying \hlink{strhypstr} in \(\ffrak{V}\mid {{\kappa ^{+\P}}}\vDash \axm{ZFC}\), by the hypothesis, \((\ffrak{V}\mid {{\kappa ^{+\P}}})\mid\delta =\ffrak{V}\mid \delta \vDash \dblq{\kappa \text{ is }\alpha \text{-hyperstrong}}\) for \emph{all} \(\delta >|\kappa +\alpha |^+\), a contradiction with the hypothesis on \(\delta _0\).  Hence \(\kappa \) must be \(\alpha \)-hyperstrong for each \(\alpha <{\kappa ^{+\P}}\), and so hyperstrong by (1).\qedhere
		\end{enumerate}
	\end{proof}
	
	This lemma allows us to show that the existence of an srs (with a strong above it) is strictly stronger than a proper class of hyperstrongs.
	
	\begin{result}\label{srsishyper}
		Let \(\kappa \) be srs such that \(\kappa ^{+\P}\) exists.  Therefore \(\kappa \) is hyperstrong.  In fact, any single \({\kappa ^{+\P}}\)-srs embedding witnesses all degrees of hyperstrength of \(\kappa \).  Moreover, in \(\ffrak{V}\mid \kappa \), there is a proper class of hyperstrongs.
	\end{result}
	\begin{proof}
		Let \(\kappa \) be srs.  \(\kappa \) is already \(0\)-hyperstrong.  There is then a \({\kappa ^{+\P}}+1\)-srs embedding \(j:\mathrm{V}\rightarrow \mathrm{M}\) with \(\cp(j)=\kappa \).  We proceed by induction on \(\alpha \) to show \(\kappa \) is \(<\alpha \)-hyperstrong in \(\ffrak{V}\) and \(\ffrak{M}\), with the base case of \(\alpha =1\) already true.  The limit case is trivial, so we consider only the successor case: showing \(\alpha \)-hyperstrength in \(\ffrak{V}\) and \(\ffrak{M}\). By \hlink{srsishypercl} (1), we may assume \(\alpha <{\kappa ^{+\P}}\).
		
		Suppose \(\kappa \) is \(<\alpha \)-hyperstrong in \(\ffrak{V}\) and \(\ffrak{M}\) so that \(\kappa \) is \(\alpha \)-hyperstrong in \(\ffrak{V}\). If there is some \(\lambda \) with no \(\lambda \)-strong extender in \(\ffrak{M}\) witnessing \(\alpha \)-hyperstrength of \(\kappa \), then by reflection, some \(\ffrak{M}\mid\delta \) (correctly) satisfies there's no such extender.  By \hlink{Σ2ref}, we can assume \(\delta ,\lambda <{\kappa ^{+\P}}\) so that \(\ffrak{M}\mid\delta =\ffrak{V}\mid\delta \) also has no such extender.  \hlink{strhypstr} tells us \(\ffrak{V}\) must satisfy \(\kappa \) isn't \(\alpha \)-hyperstrong, a contradiction.
		
		We also can show that \(\kappa \) is the limit of hyperstrongs, giving \(\ffrak{V}\mid \kappa \) as a model of a proper class of hyperstrongs.  Suppose not: \(\kappa \) is the least hyperstrong above some \(\lambda \).  By elementarity, \(j(\kappa )\) is still the least hyperstrong above \(j(\lambda )=\lambda \) in \(\ffrak{M}\) but \(\ffrak{V}\mid{{\kappa ^{+\P}}}={\ffrak{M}}\mid{\kappa ^{+\P}}\vDash \dblq{\kappa \text{ is hyperstrong}}\).  So by \hlink{srsishypercl} (2), \(\kappa \) is hyperstrong in \(\ffrak{M}\).  But \(\lambda <\kappa <j(\kappa )\) contradicts that \(j(\kappa )\) is the least hyperstrong above \(\lambda \) in \(\ffrak{M}\).\qedhere
	\end{proof}

	\section{Strongs Reflecting Strongs in the Core Model}
	
	We begin with proving the \dblq{easy} direction of \hlink{maingoal} through core model techniques.  We show (2) implies (1), meaning that \(\axm{UWISS}\) with a proper class of strongs results in a proper class of srs cardinals.  Firstly, note that a Woodin cardinal implies the existence of a proper class of srs cardinals by an easy proof.
	
	\begin{lemma}\label{woodinpcsrs}
		The existence of a Woodin cardinal implies \(\Con(\axm{ZFC}\mathrel{+}\) \dblq{There is a proper class of srs cardinals}\()\).
	\end{lemma}
	\begin{proof}
		Let \(\delta \) be Woodin and \(A=\{\kappa <\delta :\kappa \text{ is strong}\}\).  By Woodin-ness and \cite{Kanamori}, there is an unbounded (in fact, stationary) set
		\[\{\kappa <\delta :\kappa \text{ is }\lambda \text{-strong reflecting }A\text{ for all }\lambda <\delta \}\subseteq \mathrm{V}\mid \delta \text{.}\]
		But this just means we get \(\lambda \)-strong extenders \(E_\lambda \in \mathrm{V}\mid \delta \) for \(\lambda <\delta \) on each \(\kappa \) in this set such that the resulting embedding \(j_\lambda :\mathrm{V}\rightarrow \mathrm{M}_\lambda \) has \(j_\lambda (A)\cap \mathrm{V}\mid \lambda =A\cap \mathrm{V}\mid \lambda \), i.e.\ \(\{\alpha <\lambda :\ffrak{M}\vDash \dblq{\alpha \text{ is strong}}\}=\{\alpha <\lambda :\alpha \text{ is strong}\}\).  Hence \(\ffrak{V}\mid \delta \) witnesses the consistency statement.\qedhere
	\end{proof}
	
	Hence we may assume without loss of generality that there is no inner model with a Woodin and thus may work with the core model \(\ffrak{K}\) below a Woodin as presented in \cite{Steelfine} and \cite{Steelcore}.  There are a few standard facts about \(\ffrak{K}\) that we will use.  
	
	\begin{lemma}\label{coreresults}
		Suppose there is no inner model with a Woodin cardinal.  Therefore the core model \(\ffrak{K}\) is such that
		\begin{enumerate}
			\item (Local definability) For every regular \(\kappa >\aleph _1\), \(\ffrak{K}^{\ffrak{H}_\kappa }=\ffrak{K}\cap \ffrak{H}_\kappa \).
			\item (Generic absoluteness) For every poset \(\bbb{P}\in \mathrm{V}\), and every \(\bbb{P}\)-generic \(G\) over \(\mathrm{V}\), \(\ffrak{K}^{\ffrak{V}}=\ffrak{K}^{\ffrak{V}[G]}\).
			\item (Initial segment condition) If \(E\) is an extender on the sequence of \(\mathrm{K}\), then for every \(\alpha <\lh(E)\), \(E\upharpoonright\alpha \in \mathrm{K}\) is on the sequence of \(\mathrm{K}\).
		\end{enumerate}
	\end{lemma}
	
	When working with \(\ffrak{K}\), it will be useful to know when an extender is on the sequence of \(\ffrak{K}\), so we make use of the following lemma.
	
	\begin{lemma}\label{calcKseq}
		Let \(E\) be a \((\kappa ,\lambda )\)-extender with strength \(\lambda =\kappa +\delta \) such that \((\kappa ^{+\delta })^{\ffrak{V}}\) is regular.  Therefore \(F=E\cap \mathrm{K}\in \mathrm{K}\) is on the sequence of \(\ffrak{K}\) and in fact, \(\ffrak{K}\vDash \) \dblq{\(F\) is a \((\kappa ,\lambda )\)-extender with \(\mathrm{str}(F)=\lambda \)}.  Moreover, for each regular \(\mu <(\kappa ^{+\delta })^{\ffrak{V}}\) in \(\ffrak{K}\), there is an extender \(F_\mu \) such that \(\ffrak{Ult}_{F_\mu }(\ffrak{K})\) and \(\ffrak{K}\) agree on \(\mathrm{H}_{\mu }\).
	\end{lemma}
	\begin{proof}
		Let \(\mathrm{M}=\Ult_E(\mathrm{V})\) so that \({}^\omega  M\subseteq M\). Let \(j:\mathrm{V}\rightarrow \mathrm{M}\) be the canonical ultrapower map.  The hypotheses of \cite{Schindler} are satisfied and hence \(\ffrak{K}^{\ffrak{M}}\) (i.e.\ \(j(\ffrak{K})\)) is an iterate of \(\ffrak{K}\).  Moreover \(j\upharpoonright \mathrm{K}=\pi ^{\mathcal{T}}\) for some normal iteration tree \(\mathcal{T}\) on \(\ffrak{K}\) of successor length where \(\pi ^\mathcal{T}\) is the iteration map for the main branch of \(\mathcal{T}\), and \(\mathcal{T}\) has last model \(\ffrak{K}^{\ffrak{M}}\).  More explicitly, let \(\lh(\mathcal{T})=\gamma +1\).  Let \(\xi \) be such that the \(\mathcal{T}\)-predecessor of \(\xi +1\) is \(0\) and \(\xi +1\) lies on the branch of \(\gamma \).  We have models \(\langle \mathcal{M}_\alpha ^\mathcal{T}:\alpha <\lh(\mathcal{T})\rangle \) with extender sequence \(\langle E_\alpha :\alpha <\lh(\mathcal{T})-1\rangle \), iteration maps \(i_{\alpha ,\beta }:\mathcal{M}_\alpha ^\mathcal{T}\rightarrow \mathcal{M}_\beta ^\mathcal{T}\) for \(\alpha \le_{\mathcal{T}}\beta \), and the following branch 
		\[\ffrak{K}=\mathcal{M}_0^\mathcal{T}\rightarrow _{E_\xi }\mathcal{M}_{\xi +1}^\mathcal{T}\rightarrow \cdots \rightarrow \mathcal{M}_\gamma ^\mathcal{T}=\ffrak{K}^{\ffrak{M}}\text{.}\]
		We now show that \(E_\xi \upharpoonright\lambda =F\). \cite{Schindler} tells us that \(i_{0,\gamma }=j\upharpoonright\mathrm{K}\). Since \(\ffrak{V}\mid\lambda =\ffrak{M}\mid\lambda \) and \(\lambda =\kappa +\delta \), \(\mathrm{H}_{\kappa ^{+\delta }}^{\ffrak{V}}=\mathrm{H}_{\kappa ^{+\delta }}^{\ffrak{M}}\) so by local definability, \(\ffrak{K}^{\ffrak{M}}\cap \mathrm{H}_{\kappa ^{+\delta }}^{\ffrak{M}} = \ffrak{K}\cap  \mathrm{H}_{\kappa ^{+\delta }}\).  It follows that \(\str(E_\xi )=\lh(E_\xi )\ge \lambda \).  By normality of \(\mathcal{T}\), \(\cp(i_{\xi +1,\gamma })>\lh(E_\xi )\ge\lambda \).  So for any given \(\langle a,A\rangle \in [\lambda ]^{<\omega }\times [\kappa ]^{|a|}\) in \(\mathrm{K}\),
		\begin{align*}
			\langle a,A\rangle \in F&\quad\text{iff}\quad  a\in j(A)=i_{0,\gamma }(A)\\
			&\quad\text{iff}\quad  i_{\xi +1,\gamma }(a)\in i_{\xi +1,\gamma }(i_{0,\xi +1}(A))\\
			&\quad\text{iff}\quad  a\in i_{0,\xi +1}(A)\quad\text{iff}\quad  \langle a,A\rangle \in E_\xi 
		\end{align*}
		Hence \(E_\xi \) and \(F\) agree up to \(\lambda \).  By the initial segment condition, \(E_\xi \upharpoonright\lambda =F\) ensures \(F\) is on the sequence of \(\ffrak{K}\).  The strength of \(F\) being \(\lambda \) in \(\ffrak{K}\) follows from \(\mathrm{str}(E_\xi )=\lh(E_\xi )\ge\lambda \) from the normality of \(\mathcal{T}\).  Moreover, restricting \(E_\xi \) appropriately yields \(F_\mu \) as in the statement, with these in \(\ffrak{K}\) again by the initial segment condition.\qedhere
	\end{proof}
	Simply put, if \(\kappa \) is \(\kappa +\delta \)-strong in \(\ffrak{V}\), then \(\kappa \) is \(<\kappa ^{+\delta }\)-strong in \(\ffrak{K}\) according to the \(\mathrm{H}\)-hierarchy: \(\ffrak{K}\vDash \) \dblq{there are extenders \(F_\mu \) such that \(\mathrm{H}_\mu \in \Ult_{F_\mu }(\mathrm{K})\) for each regular \(\mu <\kappa ^{+\delta }\)}.
	
	In particular, any \(\kappa \) strong in \(\ffrak{V}\) is strong in \(\ffrak{K}\): to get \(\kappa \) as \(\lambda \)-strong in \(\ffrak{K}\), we just need a sufficiently large \(\mu \) such that \(\mathrm{K}\mid\lambda \subseteq \mathrm{H}_\mu ^{\ffrak{K}}\).  Then \(\kappa \) being \(\mu ^+\)-strong in \(\ffrak{V}\) implies being \(\lambda \)-strong in \(\ffrak{K}\). 
	
	But we can actually say much more than \(\ffrak{V}\)-strongs being \(\ffrak{K}\)-strong assuming \(\axm{UWISS}\): any cardinal stronger than a measurable in \(\ffrak{V}\) will be strong in \(\ffrak{K}\), as we will prove below.  As a result, any \(\ffrak{V}\)-strong cardinal is a limit of \(\ffrak{K}\)-strong cardinals, for example.  Indeed, we will show that any \(\ffrak{V}\)-strong cardinal is srs in \(\ffrak{K}\).  These ideas will be central for the (2) implies (1) direction of \hlink{maingoal}.
	\begin{theorem}\label{maincore}
		(2) implies (1) where these stand for
		\begin{enumerate}
			\item \(\Con(\axm{ZFC}+\dblq{\text{There is a proper class of strongs reflecting strongs}})\);
			\item \(\Con(\axm{ZFC}+\dblq{\text{There is a proper class of strong cardinals}}+\axm{UWISS})\).
		\end{enumerate}
	\end{theorem}
	\begin{proof}
		We're done if there is an inner model with a Woodin cardinal, so assume otherwise.  Let \( \ffrak{V}\vDash \axm{ZFC}+\axm{UWISS}\) have a proper class of strong cardinals.  It will be useful to understand what cardinals are strong in the core model \(\ffrak{K}\).
		\begin{claim}\label{maincorecl1}
			Suppose \(\kappa \) is \(\kappa +2\)-strong in \(\ffrak{V}\).  Therefore \(\kappa \) is strong in \(\ffrak{K}\).
		\end{claim}	
		\begin{proof}
			Let \(\delta >\kappa ^+\) be arbitrarily large.  Write \(\lambda =\beth _{\delta }^+\ge\delta \).  Consider \(\bbb{P}=\mathrm{Col}(\kappa ^+,\lambda )\).  By indestructibility, \(\kappa \) is still \(\kappa +2\)-strong in \(\ffrak{V}^{\bbb{P}}\) as witnessed by some \((\kappa ,\kappa +2)\)-extender \(E\) and notice \(\lambda <(\lambda ^+)^{\ffrak{V}}=(\kappa ^{++})^{\ffrak{V}^{\bbb{P}}}\).  By generic absoluteness, \(\ffrak{K}^{\ffrak{V}^{\bbb{P}}}=\ffrak{K}^{\ffrak{V}}\) and so \(\lambda \) is regular in \(\ffrak{K}\), \(\ffrak{Ult}_F(\ffrak{K})\).  By \hlink{calcKseq}, there is an extender \(F=F_{\lambda }\in \mathrm{K}\) such that \(\mathrm{K}\) and \(\Ult_F(\mathrm{K})\) agree on \(\mathrm{H}_{\lambda }\).  As \(\beth _\delta <\lambda \) is a strong limit in \(\ffrak{V}\), this also holds in \(\ffrak{K}\), so that \(\ffrak{K},\ffrak{Ult}_F(\ffrak{K})\vDash \) \dblq{\(\mathrm{H}_\lambda \supseteq \mathrm{V}_\delta \)}.  This means \(\ffrak{K}\mid\delta =\ffrak{Ult}_{F}(\ffrak{K})\mid\delta \) so that \(\kappa \) is \(\delta \)-strong in \(\ffrak{K}\).\qedhere
		\end{proof}	
		We now aim to show any \(\kappa \) strong in \(\ffrak{V}\) is actually srs in \(\ffrak{K}\). Let \(\kappa \) be strong in \( \ffrak{V}\) and let the \(\lambda \)-strength of \(\kappa \), for some strong \(\lambda \ge\kappa ^{+\P}\), be witnessed by a \((\kappa ,\lambda )\)-extender \(E\in \mathrm{V}\) and embedding \(j_E:\mathrm{V}\rightarrow \Ult_E(\mathrm{V})\).  Note that \(j_E\) restricts down to \(j_E\upharpoonright\mathrm{K}:\mathrm{K}\rightarrow \mathrm{K}^{\ffrak{Ult}_E(\ffrak{V})}\). Consider \(F=E\cap \mathrm{K}\) which is in \(\mathrm{K}\) and on the \(\ffrak{K}\)-sequence by \hlink{calcKseq}.  So consider the resulting ultrapower \(\ffrak{Ult}_{F}(\ffrak{K})\) via \(j_{F}:\mathrm{K}\rightarrow \Ult_{F}(\mathrm{K})\) which then factors \(j_E=k_F\circ j_F\) via \(k_{F}:\Ult_{F}(\mathrm{K})\rightarrow \mathrm{K}^{\ffrak{Ult}_E(  \ffrak{V})}\) as seen below.
		
		\begin{figure}[ht]
			\begin{center}
				\begin{tikzcd}[column sep=tiny]
					\mathrm{K} \arrow[rr, "j_E\upharpoonright\mathrm{K}"] \arrow[rd, "j_F"'] &                                & \mathrm{K}^{\Ult_E(\mathrm{V})} \\
					& \Ult_F(\mathrm{K}) \arrow[ru, "k_F"'] &                    
				\end{tikzcd}
			\end{center}
			\caption{Factoring ultrapower embeddings with \(\mathrm{K}\)}
		\end{figure}\label{corefactor}
		
		Note that \(\cp(k_{F})\ge\lambda \) so that for any cardinal \(\xi <\lambda \),
		\[\ffrak{Ult}_{F}(\ffrak{K})\vDash \dblq{\xi \text{ is strong}}\quad\text{iff}\quad \ffrak{K}^{\ffrak{Ult}_E(\ffrak{V})}\vDash \dblq{k_{F}(\xi )=\xi \text{ is strong}}\text{.}\]
		Thus it suffices to show \(\mathrm{K}\) and \(\mathrm{K}^{\Ult_E(\mathrm{V})}\) agree on strongs below \(\lambda \), because then all three would agree and so \(F\) would witness a \(\lambda \)-srs embedding for \(\kappa \) in \(\mathrm{K}\).  By local definability, \(\ffrak{K}\mid\lambda =\ffrak{K}^{\ffrak{Ult}_E(\ffrak{V})}\mid \lambda \) and hence we get that \(\ffrak{K}\) and \(\ffrak{K}^{\ffrak{Ult}_E(\ffrak{V})}\) agree on \(\dblq{\xi \text{ is }<\lambda \text{-strong}}\) whenever \(\xi <\lambda \).  So it suffices to show
		\[\ffrak{K},\ffrak{K}^{\ffrak{Ult}_E(\ffrak{V})}\vDash \dblq{\forall \xi <\lambda \ (\xi \text{ is }<\lambda \text{-strong}\rightarrow \xi \text{ is strong})}\text{.}\tag{\(*\)}\]
		Since \(\lambda \) is strong (in \(\ffrak{V}\)), \(\lambda \) is a limit of cardinals \(\xi \) that are \(\xi +2\)-strong.  By the strength of \(E\), if \(\xi \) is \(\xi +2\)-strong in \(\ffrak{V}\), then \(\xi \) is \(\xi +2\)-strong in \(\ffrak{Ult}_E(\ffrak{V})\), and hence \(\lambda \) is a limit of cardinals \(\xi \) that are \(\xi +2\)-strong in \(\ffrak{Ult}_E(\ffrak{V})\).  So by \hlink{maincorecl1}, \(\lambda \) is a limit of \(\ffrak{K}\)-strongs and \(\ffrak{K}^{\ffrak{Ult}_E(\ffrak{V})}\)-strongs.  As a result, if \(\xi <\lambda \) is not \(\beta \)-strong in either \(\ffrak{K}\) or \(\ffrak{K}^{\ffrak{Ult}_E(\ffrak{V})}\) for some \(\beta \), then by \hlink{Σ2ref}, this is reflected by some strong below \(\lambda \) in \(\ffrak{K}\) or \(\ffrak{K}^{\ffrak{Ult}_E(\ffrak{V})}\), showing (\(*\)) holds.  It follows that \(\ffrak{K}\) and \(\ffrak{K}^{\ffrak{Ult}_E(\ffrak{V})}\) agree on strongs below \(\lambda \), and since \(\ffrak{K}^{\ffrak{Ult}_E(\ffrak{V})}\) and \(\ffrak{Ult}_F(\ffrak{K})\) agree on strongs below \(\lambda \), we get that \(F\in \mathrm{K}\) witnesses that \(\kappa \) is \(\lambda \)-srs. A proper class of strongs in \(\ffrak{V}\) then gives a proper class of srs cardinals in \(\ffrak{K}\).\qedhere
	\end{proof}

	\section{The Forcing Direction} \label{The Forcing Direction}

	Now we show the harder direction of \hlink{maingoal}.  We show we can force a proper class of strongs with \(\axm{UWISS}\) from a proper class of srs cardinals. The general idea behind the poset, as with most indestructibility results, is a trial by fire to kill all degrees of strength.
	
	The result will be that the srs cardinals remain strong after forcing with the preparation, and small degrees of strength are, by virtue of surviving the trial by fire, weakly indestructible.  We do not get universal indestructibility for \emph{all} degrees of strength both because that's impossible by \hlink{hamkinsoopsy} and because it's possible for the tail poset to resurrect degrees of strength that were destroyed, something avoided in \cite{aptersargsyan} and \cite{apterhamkins} by cutting off the universe and declaring success anytime this might happen.
	
	Our trial proceeds with appropriate posets via a lottery in an Easton iteration, that is to say \dblq{reverse} Easton in the sense of many of Hamkins' papers and in \cite{aptersargsyan}.\footnote{I call an iteration's support \emph{Easton} iff direct limits are taken at (weakly) inaccessible stages and inverse limits elsewhere.}  What posets are appropriate?  Well, the ones that destroy the strength of a cardinal \(\kappa \) and also are simultaneously \(<\kappa \)-strategi\-cally closed and \(\le\kappa \)-distributive, basically violating \(\axm{UWISS}\).
	
	\begin{definition}\label{appropriatedef}
		For \(\delta \) a cardinal of strength \(\ge\rho \), we say a poset \(\bbb{Q}\) is \emph{\(\delta ,\rho \)-appro\-priate} iff
		\begin{enumerate}
			\item  \(\bbb{Q}\) is \(<\delta \)-strategi\-cally closed;
			\item  \(\bbb{Q}\) is \(\le\delta \)-distributive; and
			\item  \(\bbb{Q}\) destroys the \(\rho \)-strength of \(\delta \).
		\end{enumerate}
		If just (1) and (2) hold, we say \(\bbb{Q}\) is \emph{\(\delta \)-appro\-priate}.
	\end{definition}
	
	So the \emph{destructibility} of a cardinal \(\delta \)'s degree of strength \(\rho \) by a \(<\delta \)-strategi\-cally closed and \(\le\delta \)-distributive poset is obviously equivalent to the existence of a  \(\delta ,\rho \)-appro\-priate poset.	Hence we can restate \(\axm{UWISS}\) as the lack of any \(\delta ,\delta +2\)-appro\-priate posets for any \(\delta \).  We will show that the existence of appropriate posets is equivalent to the existence of \dblq{small} appropriate posets, basically meaning that we can bound the rank of \(\bbb{Q}\) and \(\rho \) by \(\delta ^{+\P}\).
	
	In examining \hlink{appropriatedef}, we have the following easy, useful result.
	
	\begin{corollary}\label{appropriatereflect}
		Let \(\bbb{P}\in \mathrm{V}\mid \alpha \) be a poset.  Therefore
		\begin{itemize}
			\item \(\bbb{P}\) is \(\delta ,\rho \)-appro\-priate iff for some (any) \(\lambda \ge|\alpha +\rho |^+\), \(\ffrak{V}\mid {\lambda }\mathrel{\vDash }\) \dblq{\(\bbb{P}\) is \(\delta ,\rho \)-appro\-priate}.
			\item For \(\dot{\bbb{Q}}\) a \(\bb{P}\)-name for a poset, \(\bbb{P}\Vdash \dblq{\bbb{Q}\text{ is }\check \delta ,\check \rho \text{-appro\-priate with rank }\check \beta }\) iff \(\ffrak{V}\mid \lambda \) satisfies this for \(\lambda \) greater than \(|\max(\alpha ,\rho ,\beta )|^+\).
		\end{itemize}
	\end{corollary}
	\begin{proof}
		It suffices to show the second since the first follows from it in the case that \(\bbb{P}\) is trivial.  Let \(\kappa =|\max(\alpha ,\rho ,\beta )|^+\). \(\bb{Q}\) is forced to be a subset of \((\mathrm{V}\mid \beta )^{\bb{P}}\).  So without loss of generality, \(\dot{\bb{Q}}\) is a nice \(\bb{P}\)-name for a subset of \((\mathrm{V}\mid \beta )^{\bb{P}}\) which therefore has rank \(<\kappa \).  \(<\delta \)-strategic closure and \(\delta \)-distributivity only make claims about \(\mathcal{P}(\bb{Q})\).  Nice names again mean that we only need access to \(\bb{P}\)-names of rank \(<\kappa \), meaning these concepts are absolute between \(\ffrak{V}\) and \(\ffrak{V}\mid \lambda \) for \(\lambda \ge\kappa \).  So the remaining concepts we need absoluteness for are related to the (non)-existence of extenders and inaccessibles: we need
		\[\bbb{P}\Vdash \dblq{\dot{\bbb{Q}}\Vdash \dblq{\check \delta \text{ is no longer }\check \rho \text{-strong}}}\text{.}\]
		Working in \(\mathrm{V}^{\bbb{P}}\), nice \(\bb{Q}\)-names for \((\delta ,\rho ')\)-extenders will have rank \(<\rho '+\beta +\omega <\kappa \).  So in \(\mathrm{V}\), we only need to consider nice names for subsets of \((\mathrm{V}\mid {\rho +\beta +\omega })^{\bb{P}}\), which all have rank \(<\kappa \).\qedhere
	\end{proof}
	The search for a \(\delta ,\rho \)-appro\-priate poset—equivalently the weak destructibility of \(\delta \)'s strength—can be bounded in the presence of lots of strong cardinals: if there is some \(\delta ,\rho \)-appro\-priate poset, then we can choose \(\rho \) and the rank of the poset to be less than \(\delta ^{+\P}\) just as in \cite{aptersargsyan}.   This isn't too difficult to show, and helps us in defining our forcing preparation later.
	\begin{lemma}\label{appropriatesmall}
		Let \(\delta \in \mathrm{Ord}\).  Let \(\bbb{R}\in \mathrm{V}\mid \alpha \) be a poset.  Suppose there's a  \(\delta ,\rho ^{**}\)-appro\-priate \(\bbb{Q}^{**}\in \mathrm{V}^{\bbb{R}}\).  Therefore, there's a \(\delta ,\rho \)-appro\-priate \(\bbb{Q}\in \mathrm{V}^{\bbb{R}}\mid (\max(\delta ,\alpha )^{+\P})^{\ffrak{V}}\) where \(\rho <(\max(\delta ,\alpha )^{+\P})^{\ffrak{V}}\) and \(\rho \le\rho ^{**}\).
	\end{lemma}
	\begin{proof}
		Let \(\bbb{Q}^{**}\) have rank greater than \(\gamma =(\max(\delta ,\alpha )^{+\P})^{\ffrak{V}}\) in \( \ffrak{V}^{\bbb{P}_\delta }\). \hlink{appropriatereflect} tells us for sufficiently large \(\lambda \), \(\bbb{Q}^{**}\) is \(\delta ,\rho ^{**}\)-appro\-priate in \(\ffrak{V}^{\bbb{R}}\mid \lambda \).  More precisely, we get a name \(\dot{\bbb{Q}}^{**}\in \mathrm{V}\mid \lambda \) such that
		\[\ffrak{V}\mid \lambda \vDash \dblq{\bbb{R}\Vdash \dblq{\dot{\bbb{Q}}^{**}\text{ is }\check\delta ,\check\rho ^{**}\text{-appro\-priate}}}\text{.}\]
		Let \(j:\mathrm{V}\rightarrow \mathrm{M}\) witness the \(\lambda \)-strength of \(\gamma \) in \( \ffrak{V}\). It follows that
		\[\ffrak{V}\mid \lambda =\ffrak{M}\mid \lambda \vDash \dblq{\bbb{R}\Vdash \dblq{\dot{\bbb{Q}}^{**}\text{ is }\check{\delta },\check{\rho }^{**}\text{-appro\-priate}}}\text{.}\]
		Since \(\rho ^{**}\) and the rank of \(\bbb{Q}^{**}\) are below \(\lambda \le j(\gamma )\), \(\ffrak{M}\) believes there's an \(\bb{R}\)-name for a poset \(\dot{\bbb{Q}}^*\) in \(\mathrm{M}\mid {j(\gamma )}\) that is \(\delta ,\rho ^*\)-appro\-priate for some \(\rho ^*<j(\gamma )\) and in particular, \(\rho ^*=\rho ^{**}\le j(\rho ^{**})\).  Elementarity then gives a name \(\dot{\bbb{Q}}\) for a  \(\delta ,\rho \)-appro\-priate poset in \(\mathrm{V}\mid \gamma \) with \(\rho <\gamma \) and \(\rho \le \rho ^{**}\).  The rank of this poset in \( \ffrak{V}^{\bbb{R}}\) is therefore below \(\gamma \).\qedhere
	\end{proof}
	In particular, if \(\alpha <\delta ^{+\P}\) then \(\dot{\bbb{Q}}\) will have rank \(<\delta ^{+\P}\).
	
	\subsection{Forcing \axm{UWISS}} \label{Forcing UWISS}
	
	We now attempt to prove the following, one of the directions from \hlink{maingoal}.
	
	\begin{theorem}\label{mainforcing}
		(1) implies (2) where these stand for
		\begin{enumerate}
			\item \(\Con(\axm{ZFC}\mathrel{+}\) \dblq{There is a proper class of srs cardinals}\()\).
			\item \(\Con(\axm{ZFC}\mathrel{+}\) \dblq{There is a proper class of strongs} \(\mathrel{+}\axm{UWISS})\).
		\end{enumerate}
	\end{theorem}
	
	Suppose \( \ffrak{V}\vDash \axm{ZFC}+\dblq{\text{there is a proper class of srs cardinals}}\).  Without loss of generality, also assume that \( \ffrak{V}\vDash \axm{GCH}\).\footnote{This can be done if necessary by forcing with the Easton support iteration that adds a cohen subset to each successor cardinal.  This will preserve srs cardinals (and many other large cardinal properties).}
	
	The basic idea is a trial by fire where anything that emerges with some degree of strength should have its degree of strength weakly indestructible.  A slight hiccup with this is that we don't actually have much control over what the resulting strength is, so rather than all of its degrees of strength being ewakly indestructible, just its small ones are.  A broad outline of the proof is as follows.  Here \(\bbb{P}_\kappa =\bigast_{\alpha <\kappa }\dot{\bbb{Q}}_\alpha \), the whole preparation is \(\bbb{P}=\bbb{P}_{\mathrm{Ord}}\), and the tail forcings from \(\delta \) up to \(\gamma \) are \(\dot{\bbb{R}}_{[\delta ,\gamma )}\cong\bigast_{\delta \le\alpha <\gamma }\dot{\bbb{Q}}_\alpha \).
	\begin{enumerate}
		\item[i.] Show that if \(\kappa \) is \(\kappa +2\)-strong in \(\ffrak{V}^{\bbb{P}}\), then this is weakly indestructible, showing \(\ffrak{V}^{\bbb{P}}\vDash \axm{UWISS}\).  This is done by showing
		\begin{itemize}
			\item if we are able to destroy \(\kappa \)'s \(\kappa +2\)-strength, then \(\dot{\bbb{Q}}_\kappa \) should be non-empty and would have already destroyed this.
		\end{itemize}
		\item[ii.] Show that if \(\kappa \) is \(\lambda \)-srs in \(\ffrak{V}\) then \(\kappa \) is \(\lambda \)-strong in \(\ffrak{V}^{\bbb{P}}\), implying that any \(\ffrak{V}\)-srs cardinal is \(\ffrak{V}^{\bbb{P}}\)-strong, and there is a proper class of both. This is done by showing \(\kappa \) is \(\lambda \)-strong in \(\ffrak{V}^{\bbb{P}_\lambda }\) for large \(\lambda \) as follows.
		\begin{itemize}
			\item Take a \(\lambda \)-srs embedding \(j:\mathrm{V}\rightarrow \mathrm{M}\) and in \(\mathrm{V}[G]=\mathrm{V}^{\bbb{P}_\lambda }\), find a generic \(H\) for the tail forcing \(\dot{\bbb{R}}_{[\lambda ,j(\lambda ))}\) over a specific hull \(\mathrm{N}[G]\preccurlyeq\mathrm{M}[G]\).
			\item Then we show this is generic over \(\mathrm{M}[G]\) using the strategic closure of the tail \emph{within} \(\ffrak{M}[G]\), not \(\ffrak{V}[G]\), to find representations for dense sets of \(\ffrak{M}[G]\) inside \(\mathrm{N}[G]\).
			\item Then we lift \(j\) to the \(\lambda \)-strong \(j^+:\mathrm{V}[G]\rightarrow \mathrm{M}[G*H]\) within \(\ffrak{V}[G]\), showing \(\kappa \) is still \(\lambda \) strong in \(\ffrak{V}[G]\).
		\end{itemize}
	\end{enumerate}
	A necessary consequence of \hlink{hamkinsoopsy} is that srs cardinals have their strengths destroyed by \(\bbb{P}_{\kappa +1}\), but will have their \(\lambda \)-strength resurrected by stage \(\lambda \) (whenever \(\lambda \) is a non-measurable, inaccessible limit of strongs).
	
	More generally, showing that degrees of strength aren't resurrected is not as simple a task as it might seem, and disregarding \hlink{hamkinsoopsy}, this is partly why we can't use this preparation to get indestructible strength for large strengths.  Let me describe the situation in a little more detail to show what goes wrong. The basic idea is that these cardinals are not going through this \dblq{trial by fire} alone, and a later cardinal can act like a medic.
	
	Naïvely, one will proceed destroying as much strength as possible: if \(\kappa \) is \(\rho \)-strong in \(\ffrak{V}^{\bbb{P}_\kappa }\), we try to destroy this with \(\kappa \)-appro\-priate posets if we can, and this constitutes the forcing done at stage \(\kappa \).  If we destroy \(\kappa \)'s strength down to be \(<\rho \), but some non-trivial forcing is done at stage \(\delta <\rho \), we might accidentally resurrect \(\kappa \)'s \(\rho \)-strength, as below with \hyperlink{interactresurrect}{Figure 2}.
	
	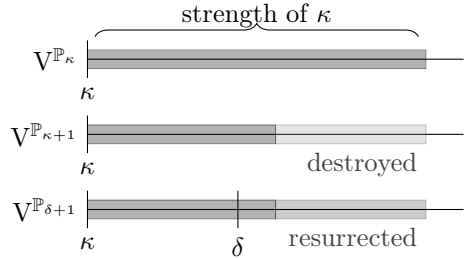
\begin{figure}[ht]
		\centering
		\begin{tikzpicture}
			\draw [decorate, decoration = {brace, amplitude=5pt}] (0.1,.3) --  (4.4,.3) node[pos=0.5,above]{strength of \(\kappa \)};
			
			\draw (0,0) node[left]{\(\mathrm{V}^{\bbb{P}_\kappa }\)} -- (5,0);
			\draw (0,.25) -- (0,-.25) node[below]{\(\kappa \)};
			\draw[fill,opacity=.3] (0,.125) -- (0,-.125) -- (4.5,-.125) -- (4.5,.125) -- (0,.125);
			
			\draw (0,-1) node[left]{\(\mathrm{V}^{\bbb{P}_{\kappa +1}}\)} -- (5,-1);
			\draw (0,-1+.25) -- (0,-1-.25) node[below]{\(\kappa \)};
			\draw[fill,opacity=.3] (0,-1+.125) -- (0,-1-.125) -- (2.5,-1-.125)  -- (2.5,-1+.125) -- (0,-1+.125);
			
			\draw[fill=black!50,opacity=.2] (2.5,-1+.125) -- (2.5,-1-.125) -- (4.5,-1-.125) node[below left,opacity=.75]{destroyed}-- (4.5,-1+.125) -- (0,-1+.125);
			
			\draw (0,-2) node[left]{\(\mathrm{V}^{\bbb{P}_{\delta +1}}\)} -- (5,-2);
			\draw (0,-2+.25) -- (0,-2-.25) node[below]{\(\kappa \)};
			\draw (2,-2+.25) -- (2,-2-.25) node[below]{\(\delta \)};
			
			\draw[fill,opacity=.3] (0,-2+.125) -- (0,-2-.125) -- (2.5,-2-.125) -- (2.5,-2+.125) -- (0,-2+.125);
			
			\draw[fill=black,opacity=.2] (2.5,-2+.125) -- (2.5,-2-.125) -- (4.5,-2-.125) node[below left,opacity=.75]{resurrected}-- (4.5,-2+.125) -- (0,-2+.125);
		\end{tikzpicture}
		\caption{Interaction resurrecting strength}
	\end{figure}\label{interactresurrect}
	
	Moreover, there is a problem if a cardinal's strength is merely playing dead. For suppose \(\kappa \) is \(\lambda \)-strong in \(\ffrak{V}\). As we approach \(\kappa \), \(\kappa \) might reduce its strength so that \(\kappa \) is merely \(\rho \)-strong in \(\ffrak{V}^{\bbb{P}_\kappa }\) and weakly indestructible there.  In that case, we would do nothing, for we don't see any strength to destroy.  But then, it may be that \(\rho \) is large enough that the next non-trivial stage of forcing occurs at some \(\delta <\rho \), in which case, \(\kappa \) might \dblq{wake up} and we accidentally resurrect \(\kappa \)'s \(\lambda \)-strength. Because we never think to return to previously dealt with stages, \(\kappa \) might remain \(\lambda \) strong in \(\ffrak{V}^{\bbb{P}}\) but not be weakly indestructible.
	
	Any attempted solution to this will ultimately either result in no resulting large strength or further forcing that might again resurrect strength by \hlink{hamkinsoopsy}. For example, we might try to collapse destroyed degrees of strength below the next stage of forcing.  This is due to the general fact that if \(\kappa \)'s strength is \(\rho <\kappa ^{+\sharp\sharp}\), then the subsequent stages won't resurrect degrees of strength by distributivity of the tail forcing proven with \hlink{prepmainforcingtail}, expressed with \hyperlink{nointeractresurrect}{Figure 3} below.
	
	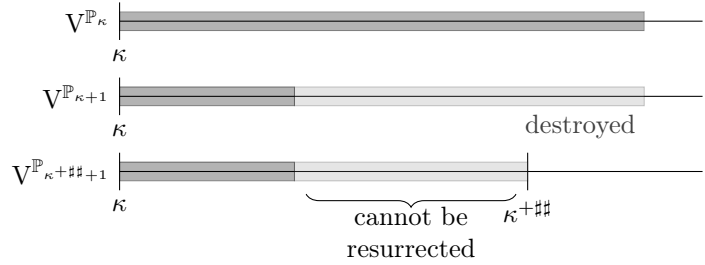
\begin{figure}[ht]
		\centering
		\begin{tikzpicture}[xscale=1.55]
			\draw (0,0) node[left]{\(\mathrm{V}^{\bbb{P}_\kappa }\)} -- (5,0);
			\draw (0,.25) -- (0,-.25) node[below]{\(\kappa \)};
			\draw[fill,opacity=.3] (0,.125) -- (0,-.125) -- (4.5,-.125) -- (4.5,.125) -- (0,.125);
			
			\draw (0,-1) node[left]{\(\mathrm{V}^{\bbb{P}_{\kappa +1}}\)} -- (5,-1);
			\draw (0,-1+.25) -- (0,-1-.25) node[below]{\(\kappa \)};
			\draw[fill,opacity=.3] (0,-1+.125) -- (0,-1-.125) -- (1.5,-1-.125)  -- (1.5,-1+.125) -- (0,-1+.125);
			
			\draw[fill=black!50,opacity=.2] (1.5,-1+.125) -- (1.5,-1-.125) -- (4.5,-1-.125) node[below left,opacity=.75]{destroyed}-- (4.5,-1+.125) -- (0,-1+.125);
			
			\draw (0,-2) node[left]{\(\mathrm{V}^{\bbb{P}_{\kappa ^{+\sharp\sharp}+1}}\)} -- (5,-2);
			\draw (0,-2+.25) -- (0,-2-.25) node[below]{\(\kappa \)};
			\draw (3.5,-2+.25) -- (3.5,-2-.25) node[below]{\(\kappa ^{+\sharp\sharp}\)};
			
			\draw[fill,opacity=.3] (0,-2+.125) -- (0,-2-.125) -- (1.5,-2-.125) -- (1.5,-2+.125) -- (0,-2+.125);
			
			\draw[fill=black!50,opacity=.2] (1.5,-2+.125) -- (1.5,-2-.125) -- (3.5,-2-.125) -- (3.5,-2+.125) -- (0,-2+.125);
			
			\draw [decorate, decoration = {brace, amplitude=5pt,mirror}] (1.5+.1,-2-.3) --  (3.5-.1,-2-.3) node[pos=0.5,below]{\(\begin{array}{c}\text{cannot be}\\\text{resurrected}\end{array}\)};
		\end{tikzpicture}
		\caption{Non-interaction leaving strength destroyed}
	\end{figure}\label{nointeractresurrect}
	
	The forcing at \(\kappa ^{+\sharp\sharp}\) shouldn't affect smaller degrees of strength by distributivity.  But the collapse we're considering done at stage \(\kappa \) \emph{could} then resurrect degrees of strength.  So this gives the following result, the proof of which is given later with \hlink{cohenresurrectdegreeslater}.
	
	\begin{result}\label{cohenresurrectdegrees}
		It's consistent (relative to two srs cardinals and proper class of strong cardinals) that we can force a \(\kappa \) that is not \(\rho \)-strong to be \(\rho \)-strong by a collapse \(\ffrak{Col}(\kappa ^+,\rho )\) 
	\end{result}
	
	This problem isn't an issue when we only care about finding \emph{one} strong or supercompact cardinal with indestructibility properties: if we ever have a measurable between a cardinal and its remaining strength, we could have just cut off the universe and end up in the desired model; no need to deal with the resurrection.  This is essentially the argument given in \cite{aptersargsyan} and \cite{apterhamkins}.  Unfortunately for us, we can't just stop our preparation and declare success, and this is partially why resurrected degrees of strength remain a fact of life.
	
	Our partial ordering is a modification of the one found in the proof of \cite{aptersargsyan}.  Specifically, we define an iteration (of proper class length) \(\bigast_{\xi \in \mathrm{Ord}}\dot{\bbb{Q}}_\xi \) which begins by adding a Cohen subset of \(\omega \) to make use of gap forcing arguments.  In other words, \(\bbb{P}_{\omega +1}\cong\ffrak{Add}(\omega ,1)\).  All other nontrivial stages of forcing can only occur at inaccessible \(\delta \) that are \(\delta +2\)-strong in \(\ffrak{V}\).

	\begin{definition}[The Preparation]\label{theprepdef}
		Suppose \(\bbb{P}_\delta =\bigast_{\xi <\delta }\dot{\bbb{Q}}_\xi \) has been defined for \(\delta >\omega \). We aim to define \(\dot{\bbb{Q}}_\delta \).
		\begin{enumerate}
			\item If no \(p\in \bb{P}_\delta \) forces that \(\delta \) is inaccessible, then \(\dot{\bbb{Q}}_\delta =\dot{\bbone}\) is trivial.
			\item Otherwise, suppose \(\delta \) is forced by some \(p\in \bb{P}_\delta \) to be \(<\lambda \)-strong for some (maximal) \(\lambda \) (allowing \(\lambda =0\) if \(\delta \) isn't measurable, and \(\lambda =\mathrm{Ord}\) if strong), and work below \(p\).
			\begin{enumerate}
				\item If the \(<\lambda \)-strength of \(\delta \) is weakly indestructible via posets of rank \(<(\delta ^{+\P})^{\ffrak{V}}\), then define \(\dot{\bbb{Q}}_\delta =\dot{\bbone}\).
				\item Otherwise, we let \(\rho <\lambda \) be the minimal degree of \(\delta \)'s strength that is weakly destructible and let \(\dot{\bbb{Q}}_\delta \) be the lottery sum of all (what are forced to be) posets that take the form \(\dot{\bbb{B}}*\dot{\bbb{C}}\) where the following happen.
				\item \(\dot{\bbb{B}}\) is an \(\delta ,\rho \)-appro\-priate poset of rank \(<(\delta ^{+\P})^{\ffrak{V}}\), and
				\item In \(\ffrak{V}^{\bbb{P}_\delta *\dot{\bbb{B}}}\), if  \(\rho <\delta ^{+\sharp\sharp}\le |\dot{\bb{B}}|\), then \(\dot{\bbb{C}}\) is a name for \(\ffrak{Col}(\rho ^{+\sharp},|\dot{\bb{B}}|)\).  Otherwise \(\dot{\bbb{C}}\) is trivial.
			\end{enumerate}
		\end{enumerate}
		Using Easton support for limit stages, we write \(\bbb{P}=\bigast_{\xi \in \mathrm{Ord}}\dot{\bbb{Q}}_\xi \) for the class iteration and \(\dot{\bbb{R}}_{[\delta ,\lambda )}\cong\bigast_{\delta \le\xi <\lambda }\dot{\bbb{Q}}_\xi \) for the tail forcing of \(\bbb{P}_\lambda \) whenever \(\delta <\lambda \).
	\end{definition}
	
	Two remarks about this preparation: firstly, note that each \(\dot{\bbb{Q}}_\delta \) is  \(\delta ,\rho \)-appro\-priate for some \(\rho \) whenever it's non-trivial.  Secondly, it's not hard to see that we can regard \(\bbb{P}_\delta \subseteq \mathrm{V}\mid\delta ^{+\P}\) for any \(\delta \).  In fact, we can regard \(\bb{P}_\delta \subseteq \mathrm{V}\mid\delta \) whenever \(\dot{\bbb{Q}}_\delta \) is non-trivial, since \(\delta \) is therefore inaccessible and not collapsed by any previous stage, meaning every previous stage had a smaller cardinality and thus smaller rank.
	
	The collapsing poset used for each poset in the lottery is used to give better control over the tail forcing, and in particular to show \(\bbb{P}_{\delta }\) is \(\delta \)-cc whenever \(\delta \) is still inaccessible there.  Moreover, the collapse allows us to ensure that once we collapse a \emph{small} degree of strength with a small poset, the strength stays collapsed and won't be resurrected.  Larger degrees, however, we make no promises about.
	
	There are several standard facts of tail forcings we need to use in order to prove that the tail forcings of the preparation are strategically closed.  The use of this will be in showing that \(\dot{\bbb{R}}_{[\delta ,\lambda )}\) is \(\delta \)-appro\-priate and thus can be used in arguments with \(\delta ,\rho \)-appro\-priate posets since it is close to being one.  One background result used is the following~\cite{cummings}.
	
	\begin{theorem}\label{iterationtailclosure}
		Let \(\bbb{P}_\lambda '=\bigast_{\xi <\lambda }\dot{\bbb{Q}}_\xi '\) be a \(\lambda \)-stage iteration such that for some \(\delta <\lambda \) and some \(\mu \),
		\begin{enumerate}
			\item inverse limits or direct limits are taken at every limit stage;
			\item inverse limits are taken at every limit stage \(\ge\delta \) of cofinality \(<\mu \);
			\item \(\dot{\bbb{Q}}_\xi '\) is (forced to be) \(<\mu \)-strategi\-cally closed for every \(\xi \) with \(\delta \le \xi <\lambda \); and
			\item \(\bbb{P}_\delta '\) is \(\mu \)-cc.
		\end{enumerate}
		Therefore the tail forcing \(\dot{\bbb{R}}_{[\delta ,\lambda )}\) is (forced to be) \(<\mu \)-strategi\-cally closed.
	\end{theorem}
	
	\begin{corollary}\label{prepmainforcingtail}
		Let \(\delta <\lambda \).  Therefore, the tail forcing of our preparation, \(\dot{\bbb{R}}_{[\delta ,\lambda )}\) is (forced to be) \(<\delta \)-strategi\-cally closed and \(\le\delta \)-distributive.
	\end{corollary}
	\begin{proof}
		If \(\dot{\bbb{Q}}_\delta \) is non-trivial, then below the appropriate conditions, \(\delta \) is inaccessible in \(\ffrak{V}^{\bbb{P}_\delta }\) so we take a direct limit at stage \(\delta \).  It follows that \(\delta \) was not collapsed at some previous stage by a \(\kappa ,\rho \)-appro\-priate poset \(\bbb{Q}\) where \(\kappa <\delta \le\max(\rho ,|\bb{Q}|)\).  So for each \(\kappa <\delta \), \(\dot{\bbb{Q}}_\kappa \) is equivalent to a poset with rank \(<\delta \) and in fact we can regard \(\bb{P}_\delta \subseteq \mathrm{V}\mid\delta \) and so \(\bbb{P}_\delta \) is \(\delta \)-cc.  Using \hlink{iterationtailclosure}, Easton support iterations clearly take inverse or direct limits everywhere, and only direct limits at certain regular stages.  So at any \(\xi \ge\delta >\cof(\xi )\), we take an inverse limit.  Thus (1) and (2) hold from \hlink{iterationtailclosure}.  By hypothesis, (3) holds since lottery sums of \(<\kappa \)-strategi\-cally closed posets are \(<\kappa \)-strategi\-cally closed.  (4) holds since \(\bbb{P}_\delta \) is \(\delta \)-cc and thus \(\dot{\bbb{R}}_{[\delta ,\lambda )}\) is \(<\delta \)-strategi\-cally closed.
		
		If \(\dot{\bbb{Q}}_\delta \) is trivial, then below the appropriate conditions, the next non-trivial stage (if there is one) \(\dot{\bbb{Q}}_\mu \) has by the above argument that \(\dot{\bbb{R}}_{[\mu ,\lambda )}\cong\dot{\bbb{R}}_{[\delta ,\lambda )}\) is \(<\mu \)-strategi\-cally closed and hence \(<\delta \)-strategi\-cally closed.  If there is no non-trivial stage above \(\delta \), then the tail forcing is trivial and hence \(<\delta \)-strategi\-cally closed.
		
		For \(\le\delta \)-distributivity, note that \(\dot{\bbb{R}}_{[\delta ,\lambda )}\cong\dot{\bbb{Q}}_\delta *\dot{\bbb{R}}_{(\delta ,\lambda )}\) is the two-step iteration of two \(\le\delta \)-distributive posets.\qedhere
	\end{proof}
	
	Given that measurability is already indestructible, the next non-trivial stage of forcing after \(\delta \) occurs at stage \(\delta ^{+\sharp\sharp}\) at the earliest.
	
	\begin{lemma}\label{nextnontriv}
		The first non-trivial stage of forcing after any \(\kappa \in \mathrm{Ord}\) is at least \((\kappa ^{+\sharp\sharp})^{\ffrak{V}^{\bbb{P}_{\kappa +1}}}\) 
	\end{lemma}
	\begin{proof}
		To have any degree of \(\delta \)'s strength weakly destructible over \(\ffrak{V}^{\bbb{P}_\delta }\), we require \(\delta \) to be at least \(\delta +2\)-strong.  Hence the least \(\delta \) such that \(\bbb{P}_{\delta }\cong\bbb{P}_{\kappa +1}\) but \(\dot{\bbb{Q}}_\delta \) is non-trivial must have that \(\delta \) is \(\delta +2\)-strong in \(\ffrak{V}^{\bbb{P}_\delta }=\ffrak{V}^{\bbb{P}_{\kappa +1}}\) and hence \(\delta \ge(\kappa ^{+\sharp\sharp})^{\ffrak{V}^{\bbb{P}_{\kappa +1}}}\).
	\end{proof}
	
	Hence all the degrees of \(\kappa \)'s strength below \((\kappa ^{+\sharp\sharp})^{\ffrak{V}^{\bbb{P}_{\kappa +1}}}\) remain indestructible even if larger degrees are accidentally resurrected.  The basic proof of this is that if we \emph{could} destroy something, we would have destroyed it already.
	
	\begin{result}\label{mainforcingwss}
		Let \(\kappa \) be such that \(\bbb{P}_\kappa \) and \(\bbb{P}_{\kappa +1}\Vdash \) \dblq{\(\check\kappa \) is \(\check\rho \)-strong for \(\check\rho <\kappa ^{+\sharp\sharp}\)}. Therefore \(\bbb{P}\Vdash \) \dblq{\(\check\kappa \)'s \(\rho \)-strength is weakly indestructible}.
	\end{result}
	\begin{proof}
		By downward absoluteness, in \(\ffrak{V}^{\bbb{P}_{\kappa }}\), \(\kappa \) is still inaccessible.  Hence we are in case (2) of \hlink{theprepdef}.  As \dblq{small} posets and by gap forcing~\cite{GapForcing}, \(\delta ^{+\sharp}\), \(\delta ^{+\sharp\sharp}\), and \(\delta ^{+\P}\) are all the same in \(\ffrak{V}\) and \(\ffrak{V}^{\bbb{P}_\kappa }\). (Such cardinals retain their large cardinal status by small forcing, and no new such cardinals are added above \(\delta \) by \cite{GapForcing}.)
		
		As a result, because \(\rho <\kappa ^{+\sharp\sharp}\) in \(\ffrak{V}^{\bbb{P}_{\kappa +1}}\), the tail forcing after \(\kappa \) is sufficiently distributive such that the \(\rho \)-strong extender on \(\kappa \) in \(\ffrak{V}^{\bbb{P}_{\kappa +1}}\) is still \(\rho \)-strong in \(\ffrak{V}^{\bbb{P}}\).  Thus it suffices to show weak indestructibility for this degree of strength. So suppose \(\dot{\bbb{Q}}\) is \(\kappa ,\rho \)-appro\-priate in \(\ffrak{V}^{\bbb{P}}\).  By distributivity of the tail forcing, \(\dot{\bbb{Q}}\in \mathrm{V}^{\bbb{P}_\lambda }\) for some \(\lambda >\kappa \).  The tail forcing \(\dot{\bbb{R}}_{(\kappa ,\lambda )}\) is therefore \(\kappa \)-appro\-priate by \hlink{prepmainforcingtail}.
		
		This tells us that \(\dot{\bbb{Q}}_\kappa \) must be non-trivial.  To see this, otherwise \(\kappa \) must be \(\rho \)-strong in \(\ffrak{V}^{\bbb{P}_{\kappa +1}}=\ffrak{V}^{\bbb{P}_\kappa }\).  It follows that \(\dot{\bbb{R}}_{[\kappa ,\lambda )}*\dot{\bbb{Q}}\) is \(\kappa ,\rho \)-appro\-priate in \(\ffrak{V}^{\bbb{P}_\kappa }\) so that by \hlink{appropriatesmall}, we can find a \(\kappa ,\rho \)-appro\-priate poset of small size and so \(\dot{\bbb{Q}}_\kappa \) should be non-trivial, a contradiction.
		
		Since \(\dot{\bbb{Q}}_\kappa \) is non-trivial, by \hlink{theprepdef}, we forced with some poset \(\dot{\bbb{B}}*\dot{\bbb{C}}\) at stage \(\kappa \) where \(\dot{\bbb{B}}\) is \(\kappa ,\rho _\kappa \)-appro\-priate in \(\ffrak{V}^{\bbb{P}_\kappa }\) with minimal \(\rho _\kappa \). If \(\rho <\rho _\kappa \), the \(\kappa ,\rho \)-appro\-priate \(\dot{\bbb{R}}_{[\kappa ,\lambda )}*\dot{\bbb{Q}}\) would violate minimality of \(\rho _\kappa \) (via \hlink{appropriatesmall} to ensure we stay below \((\delta ^{+\P})^{\ffrak{V}}\) in \(\ffrak{V}^{\bbb{P}_{\kappa }}\)).  So \(\rho \ge\rho _\kappa \).  We now break into cases.
		\begin{itemize}
			\item Suppose \(\dot{\bbb{C}}\) is non-trivial.  Thus \(\rho _\kappa <\kappa ^{+\sharp\sharp}\le|\dot{\bbb{B}}|\) and \(\dot{\bbb{C}}\) collapses \(|\dot{\bbb{B}}|\) to be \(\rho ^{+\sharp}\) (which is \(<\kappa ^{+\sharp\sharp}\)) in a way that is \(<\rho ^{+\sharp}\)-distributive.  In particular, this preserves the lack of \(\rho _\kappa \)-strong extenders, and so the lack of \(\rho \)-strong extenders in \(\ffrak{V}^{\bbb{P}_\kappa *\dot{\bbb{B}}}\) to \(\ffrak{V}^{\bbb{P}_\kappa *\dot{\bbb{B}}*\dot{\bbb{C}}}\), contradicting that \(\kappa \) is \(\rho \)-strong in \(\ffrak{V}^{\bbb{P}_{\kappa +1}}\)
			\item Suppose \(\dot{\bbb{C}}\) is trivial so that because \(\dot{\bbb{B}}\) is \(\kappa ,\rho _\kappa \)-appro\-priate, \(\kappa \) is not \(\rho _\kappa \)-strong in \(\ffrak{V}^{\bbb{P}_{\kappa }*\dot{\bbb{B}}}=\ffrak{V}^{\bbb{P}_{\kappa +1}}\) and hence not \(\rho \)-strong there, a contradiction.\qedhere
		\end{itemize}
	\end{proof}
	
	Unfortunately, this does not guarantee weak indestructibility for \(\kappa \)'s degrees of strength \(<\kappa ^{\sharp\sharp}\) in \(\ffrak{V}^{\bbb{P}}\), but merely for strength \(<(\kappa ^{\sharp\sharp})^{\ffrak{V}^{\bbb{P}_{\kappa +1}}}\) and so in particular \(\kappa +2\)-strength, \(\kappa ^{+\sharp}\)-strength, and \(<\lambda \)-strength for the first \(\lambda \) a measurable limit of measurables in \(\ffrak{V}^{\bbb{P}}\).
	
	\begin{corollary}[Forcing \(\axm{UWISS}\)]\label{PforcesUWISS}
		Suppose \(\ffrak{V}\vDash \) \dblq{there is a proper class of strongs}.  Therefore
		\begin{itemize}
			\item the preparation \(\bbb{P}\) is well defined and \(\bbb{P}\Vdash \axm{UWISS}\).
			\item In fact, if \(\kappa \) is \(\rho \)-strong for \(\rho <(\kappa ^{+\sharp\sharp})^{\ffrak{V}^{\bbb{P}_{\kappa +1}}}\) in \(\ffrak{V}^{\bbb{P}}\) then this strength is weakly indestructible in \(\ffrak{V}^{\bbb{P}}\).
			\item In particular, any cardinal \(\kappa \)'s strength that is below the least measurable limit of measurables above \(\kappa \) is weakly indestructible, e.g.\ \(\kappa +2\)-strength, \(\kappa ^{+\sharp}\)-strength, \((\kappa ^{+\sharp})^{+\sharp}\)-strength, etc.
		\end{itemize}
	\end{corollary}
	\begin{proof}
		If \(\kappa \) is measurable in \(\ffrak{V}^{\bbb{P}}\), then this degree of strength is weakly indestructible. So suppose \(\kappa \) is stronger than a measurable, but there is some \(\dot{\bbb{Q}}\in \mathrm{V}^{\bbb{P}}\) that is \(\kappa ,\rho \)-appro\-priate for some \(\rho <(\kappa ^{+\sharp\sharp})^{\ffrak{V}^{\bbb{P}_{\kappa +1}}}\), which has some rank below an inaccessible \(\gamma >\rho ,\rank(\dot{\bbb{Q}})\).  The tail forcing \(\dot{\bbb{R}}_{[\gamma ,\infty )}\) is \(\le\gamma \)-distributive by \hlink{prepmainforcingtail} and hence \(\dot{\bbb{Q}}\in \mathrm{V}^{\bbb{P}_{\gamma }}\) and \(\ffrak{V}^{\bbb{P}_\gamma }\mid\gamma =\ffrak{V}^{\bbb{P}}\mid\gamma \).  So by \hlink{appropriatereflect}, \(\dot{\bbb{Q}}\) is \(\kappa ,\rho \)-appro\-priate in \(\ffrak{V}^{\bbb{P}_\gamma }\).  The tail forcing \(\dot{\bbb{R}}_{(\kappa ,\gamma )}\) is sufficiently distributive by \hlink{nextnontriv} to show \(\kappa \) was \(\rho \)-strong in \(\ffrak{V}^{\bbb{P}_{\kappa +1}}\) and so \(\dot{\bbb{R}}_{(\kappa ,\lambda )}*\dot{\bbb{Q}}\) is \(\kappa ,\rho \)-appro\-priate there, contradicting \hlink{mainforcingwss}.
		
		By \hlink{nextnontriv}, measurables between \(\kappa \) and \((\kappa ^{+\sharp\sharp})^{\ffrak{V}^{\bbb{P}_{\kappa +1}}}\) are preserved by the tail forcing by \hlink{nextnontriv}.  So it follows that every cardinal below the least measurable limit of measurables above \(\kappa \)—being below \((\kappa ^{+\sharp\sharp})^{\ffrak{V}^{\bbb{P}_{\kappa +1}}}\)—is indestructible.\qedhere
	\end{proof}
	
	We may not get better than this, since we could have the following: in \(\ffrak{V}^{\bbb{P}_{\kappa +1}}\), \(\kappa \) is \(\rho \)-strong for \(\rho >\kappa ^{+\sharp\sharp}\).  Then at stage \(\mu =(\kappa ^{+\sharp\sharp})^{\ffrak{V}^{\bbb{P}_{\kappa +1}}}\) we do some non-trivial forcing that destroys \(\mu \)'s \(\mu +2\)-strength, and subsequently resurrects \(\kappa \)'s \(\rho '\)-strength for \(\rho '>\mu \) in a way that the new \((\kappa ^{+\sharp\sharp})^{\ffrak{V}^{\bbb{P}}}>\rho '\) but \(\kappa \)'s \(\rho '\)-strength is now destructible.  Nevertheless, this doesn't affect somewhat small degrees of strength as the above shows.
	
	\subsection{A proper class of strongs}

	So all that remains is to show that srs cardinals of \(\ffrak{V}\) are strong in \(\ffrak{V}^{\bbb{P}}\).  We do this by working with partial degrees of reflecting strongs.  The usefulness of reflecting strongs allows us to properly calculate the preparation up to limits of strongs.
	
	\begin{lemma}\label{mainforcingcl2}
		Let \(\lambda \) be a limit of strongs. Let \(j:\mathrm{V}\rightarrow \mathrm{M}\) be at least \(\lambda \)-strong such that \(\ffrak{V}\) and \(\ffrak{M}\) agree on strongs \(<\lambda \).  Therefore \(\bbb{P}_{\lambda }^{ \ffrak{V}}=j(\bbb{P})_{\lambda }=\bbb{P}^{\ffrak{M}}_{\lambda }\).
	\end{lemma}
	\begin{proof}
		Let \(\kappa =\cp(j)\).  \(j\) will not screw with the Easton support below \(\lambda \), so it suffices to show \(\dot{\bbb{Q}}_\delta ^{\ffrak{V}}=\dot{\bbb{Q}}_\delta ^{\ffrak{M}}\) for all cardinals \(\delta <\lambda \).  Proceed by induction on \(\delta \).  In case (1) of \hlink{theprepdef}, \(\delta \)'s non-inaccessibility is easily absolute between \(\ffrak{V}\) and \(\ffrak{M}\) since inductively \(\bbb{P}_\delta =\bbb{P}_\delta ^{\ffrak{M}}\).
		
		Note that since \(\lambda \) is a limit of strongs, if \(\delta <\lambda \), then \((\delta ^{+\P})^{\ffrak{V}}=(\delta ^{+\P})^{\ffrak{M}}\) and so we can unambiguously write \(\delta ^{+\P}\) in such cases.  Note also if \(\delta \) is \(<\lambda _0\)-strong in \(\ffrak{V}\), then either \(\lambda _0=\delta ^{+\P}<\lambda \) implies \(\delta \) is strong in both since the models agree on strongs below \(\lambda \), or else \(\lambda _0<\delta ^{+\P}<\lambda \) and so the lack of \(\lambda _0\)-strong extenders in \(\ffrak{V}\mid\lambda \) matches with \(\ffrak{M}\mid\lambda \).  Note also that all (names for) collapses we consider will exist in \(\mathrm{V}\mid\delta ^{+\P}=\mathrm{M}\mid\delta ^{+\P}\) and thus have the same interpretation in both models.
		
		In case (2), by \hlink{appropriatereflect}, the existence of \(\rho <\lambda _0\le \delta ^{+\P}<\lambda \) and a \(<\delta \)-strategi\-cally closed, \(\le\delta \)-distributive \(\dot{\bbb{B}}\in \mathrm{V}\mid\delta ^{+\P}\) such that \(\delta \) isn't \(\rho \)-strong after forcing with \(\bbb{P}_\delta *\dot{\bbb{B}}\) can be calculated in \(\ffrak{V}\mid\lambda =\ffrak{M}\mid\lambda \). Hence the two share the same such posets, and moreover, the minimal \(\rho \) witnessing this is the same for both. The calculation of \(\delta ^{+\sharp\sharp}\) in both will be below \(\delta ^{+\P}\) and easily the same in both models.  The collapse (also being below \(\delta ^{+\P}\)) will also be the same, meaning \(\dot{\bbb{Q}}_\delta \) is the same in both.\qedhere
	\end{proof}
	In particular, if \(j:\mathrm{V}\rightarrow \mathrm{M}\) is a \(\lambda \)-srs embedding, then \(\bbb{P}_{\lambda }^{\ffrak{V}}=\bbb{P}_{\lambda }^{\ffrak{M}}\) whenever \(\lambda \) is a limit of strongs.  This gives us the edge over hyperstrength in generalizing \cite{aptersargsyan} which generally only gets agreement for \(\bbb{P}_{\cp(j)}\), but the resulting argument is adapted from \cite{aptersargsyan}.
	
	\begin{lemma}\label{mainforcingmainlem}
		Let \(\kappa \) be strong and \(\lambda \)-srs where \(\lambda \) is a limit of strongs.  Therefore \(\bbb{P}_\lambda \Vdash \) \dblq{\(\kappa \) is \(\lambda \)-strong}.
	\end{lemma}
	\begin{proof}
		Let \(j:\mathrm{V}\rightarrow \mathrm{M}\) be \(\lambda \)-srs with \(\cp(j)=\kappa \) so that \(\mathrm{M}=\Ult_E(\mathrm{V})\) for some \((\kappa ,\lambda )\)-extender \(E\).  We can factor by \hlink{mainforcingcl2}
		\begin{align*}
			j(\bbb{P}_{\lambda })=\bbb{P}_{j(\lambda )}^{\ffrak{M}}&\cong(\bbb{P}_\kappa *\dot{\bbb{Q}}_\kappa *\dot{\bbb{R}}_{(\kappa ,\lambda )})^{\ffrak{M}}*\dot{\bbb{R}}_{[\lambda ,j(\kappa ))}^{\ffrak{M}}*\dot{\bbb{R}}_{[j(\kappa ),j(\lambda ))}^{\ffrak{M}}\\
			&\cong(\bbb{P}_\kappa *\dot{\bbb{Q}}_\kappa *\dot{\bbb{R}}_{(\kappa ,\lambda )})^{\ffrak{V}}*\dot{\bbb{R}}_{[\lambda ,j(\kappa ))}^{\ffrak{M}}*\dot{\bbb{R}}_{[j(\kappa ),j(\lambda ))}^{\ffrak{M}}\\
			&\cong \bbb{P}_\lambda ^{\ffrak{V}}*\dot{\bbb{R}}_{[\lambda ,j(\kappa ))}^{\ffrak{M}}*\dot{\bbb{R}}_{[j(\kappa ),j(\lambda ))}^{\ffrak{M}}\text{.}
		\end{align*}
		Without loss of generality, \(\lambda \) isn't \(\lambda +2\)-strong in \(\ffrak{M}\) (otherwise just taking another ultrapower by the Mitchell-least measure on \(\lambda \)) so that in \(\ffrak{M}^{\bbb{P}_\lambda }\), \(\lambda \) is at most measurable with therefore indestructible degrees of strength and its original strength already small: \(\dot{\bbb{Q}}_\lambda ^{\ffrak{M}}=\dot{\bbone}\).  Thus \(\dot{\bbb{R}}_{[\lambda ,j(\lambda ))}^{\ffrak{M}}\) is actually \(\lambda ^+\)-strategi\-cally closed in \(\ffrak{M}\) (and indeed much more).
		
		So let \(G=G_0*G_1\) be \(\bbb{P}_\kappa *\dot{\bbb{R}}_{[\kappa ,\lambda )}\)-generic over \(\mathrm{V}\) such that \(\ffrak{V}[G]\vDash \) \dblq{\(\kappa \) is not \(\lambda \)-strong}.  Our goal is to lift \(j\) to \(j^+:\mathrm{V}[G]\rightarrow \mathrm{M}[G*H]\) in \(\ffrak{V}[G]\) for some \(H=H_0*H_1\in \mathrm{V}[G]\) that is \((\dot{\bbb{R}}_{[\lambda ,j(\kappa ))}^{\ffrak{M}}*\dot{\bbb{R}}_{[j(\kappa ),j(\lambda ))}^{\ffrak{M}})_G\)-generic over \(\mathrm{M}[G]\) such that \(j^+\) remains \(\lambda \)-strong.  This requires examining \(j"G\), and beyond this there are only a couple steps in this lift-up argument: first building \(H_0\) arbitrarily and then generating \(H_1\) from \(j"G_1\).
		
		\begin{claim}\label{scrmainforcingmainlemcl1}
			For \(j(p)\in j"G\), \(j(p)\upharpoonright[\lambda ,j(\kappa ))\)—and in fact \(j(p)\upharpoonright[\kappa ,j(\kappa ))\)—is trivial.  Hence the only non-trivial information \(j"G\) encodes occurs before \(\lambda \) and after \(j(\kappa )\).  
		\end{claim}
		\begin{proof}
			Let \(p\in G\) be arbitrary.  Since we take a direct limit at stage \(\kappa \), there is some \(\alpha <\kappa \) where \(p\upharpoonright[\alpha ,\kappa )\) is just a sequence of \(\dot{\bbone}\)s: \(p\upharpoonright[\alpha ,\kappa )\) is trivial in \(\dot{\bbb{R}}_{[\alpha ,\kappa )}\).  By elementarity, \(j(p)\) is similarly trivial from \(j(\alpha )=\alpha <\kappa \) to \(j(\kappa )\).  In particular, \(j(p)\upharpoonright[\kappa ,j(\kappa ))\) is trivial.\qedhere
		\end{proof}
		
		As a result, \emph{any} \(H_0\) that is \(\bbb{R}_{[\lambda ,j(\kappa ))}^{\ffrak{M}[G]}=(\dot{\bbb{R}}_{[\lambda ,j(\kappa ))}^{\ffrak{M}})_G\)-generic over \(\ffrak{M}[G]\) has \(j"G\upharpoonright j(\kappa )\) contained in \(G*H_0\). So we merely need to find such a generic over \(\ffrak{M}[G]\) in \(\ffrak{V}[G]\) to get \(H_0\).  Then we find \(H_1\).  \hlink{scrmainforcingmainlemcl1} is partly why we need to break up the iteration as we do.  The general idea is that \(j"G\) might have conditions with potentially unbounded support in \(j(\lambda )\).  So we must generate the generic over the end tail.  \hlink{scrmainforcingmainlemcl1} ensures that the middle is left unaffected since \(j"G\) has no real information about it.
		
		As an ultrapower, we can regard \(\mathrm{M}=\{j(f)(r):r\in [\lambda ]^{<\omega }\wedge f:[\kappa ]^{<\omega }\rightarrow \mathrm{V}\}\) so that there are elements of \(r_0,r_1\in [\lambda ]^{<\omega }\) and functions from \([\kappa ]^{<\omega }\) to \(\mathrm{V}\) that represent \(\bbb{P}_\kappa ^{\ffrak{V}}\), \(\dot{\bbb{R}}_{(\kappa ,j(\lambda ))}^{\ffrak{M}}\).  So now we consider the elementary submodel \(\ffrak{N}\preccurlyeq\ffrak{M}\) defined by
		\[\mathrm{N}=\{j(f)(r_0\cup r_1,\kappa ,\lambda ): f\colon[\kappa ]^{<\omega }\rightarrow \mathrm{V}\}\ni \bbb{P}_\kappa ^{\ffrak{V}},\bbb{P}_\lambda ^{\ffrak{V}},\dot{\bbb{Q}}_\kappa ,\dot{\bbb{R}}_{(\lambda ,j(\kappa ))}^{\ffrak{M}},\kappa ,\lambda \text{,}\]
		It's also not hard to see that \(\ffrak{V}\vDash \dblq{{}^{\kappa }\mathrm{N}\subseteq \mathrm{N}}\). We write \(\mathrm{N}[G]\) for \(\{\tau _G:\tau \in \mathrm{N}\}\).
		\begin{claim}\label{scrmainforcingcl5cl1}
			In \(\ffrak{V}[G]\), there is an \(H_0\) \(\bbb{R}_{(\lambda ,j(\kappa ))}^{\ffrak{M}[G]}\)-generic over \(\mathrm{N}[G]\).
		\end{claim}
		\begin{proof}
			\(\bbb{R}_{(\lambda ,j(\lambda ))}^{\ffrak{M}[G]}\) is \(\lambda ^+\)-strategi\-cally closed in \(\ffrak{M}[G]\) and this translates to being merely \(\kappa ^+\)-strategi\-cally closed in \(\ffrak{V}[G]\) by the closure conditions of \(\ffrak{M}\) and \(\ffrak{N}\) in \(\ffrak{V}\): \(\ffrak{N}[G]\preccurlyeq\ffrak{M}[G]\) is closed under \(\kappa \)-sequences.  Hence \(\bbb{R}_{(\lambda ,j(\lambda ))}^{\ffrak{M}[G]}\cap \mathrm{N}[G]\) is still \(\kappa ^+\)-strategi\-cally closed in \(\ffrak{V}[G]\).  Since dense sets in \(\mathrm{N}[G]\) can be identified with \(f:[\kappa ]^{<\omega }\rightarrow \mathrm{V}_\kappa \), a simple counting argument shows that in \(\ffrak{V}[G]\), there are at most \(2^{j(\kappa )}=\kappa ^+\)-many antichains of the poset in \(\mathrm{N}[G]\).  Thus we can find in \(\ffrak{V}[G]\) an \(H\) that is \(\bbb{R}_{(\kappa ,j(\kappa ))}^{\ffrak{M}[G]}\cap \mathrm{N}[G]\)-generic over \(\mathrm{N}[G]\) by standard techniques (extending into dense open sets one by one and leaving the other player in the strategic closure game to clean up our mess at limit stages).\qedhere
		\end{proof}
		Now we must show that \(H_0\) is actually generic over \(M[G]\).  Let \(D\) be a dense open set of \(\bbb{R}_{(\lambda ,j(\kappa ))}^{\ffrak{M}[G]}\).  We can represent \(D\) by \(j(f)(r)\) for some \(f:[\kappa ]^{<\omega }\rightarrow \bb{P}_{\kappa }\) and \(r\in [\lambda ]^{<\omega }\).  So in \(\ffrak{M}[G]\), consider the set \(\{j(f)(s)\text{ dense open}:s\in [\lambda ]^{<\omega }\}\).  Since \(\ffrak{M}[G]\) thinks that the tail \(\bbb{R}_{(\lambda ,j(\kappa ))}^{\ffrak{M}[G]}\) is \(\lambda ^+\)-strategi\-cally closed, the intersection of all of these sets \(E=\bigcap _{s\in [\lambda ]^{<\omega }}j(f)(s)\) is definable just with \(f\) and so in \(\mathrm{N}[G]\), meaning \(E\cap H_0\neq \emptyset \) and showing that it's generic over \(\ffrak{M}[G]\).
		
		\begin{claim}\label{scrmainforcingmainlemcl2}
			Let \(H_1\) be the filter generated by \(j"G_1\).  Therefore \(H_1\) is \(\bbb{R}_{[j(\kappa ),j(\lambda ))}^{\ffrak{M}[G*H_0]}\)-generic over \(\mathrm{M}[G*H_0]\) and as a result of \hlink{scrmainforcingmainlemcl1}, \(j"G\subseteq G*H_0*H_1\).
		\end{claim}
		\begin{proof}
			Let \(D\in \mathrm{M}[G*H_0]\) be open dense with name \(\dot D\).  It follows that we can represent \(D\) as \(j(d)(r)_{G*H_0}\) for some \(d:[\kappa ]^{<\omega }\rightarrow \mathrm{V}\) and \(r\in [\lambda ]^{<\omega }\) where each \(d(s)\) is a name for a dense open set in \(\dot{\bbb{R}}_{[\kappa ,\lambda )}\).  In \(\ffrak{V}[G_0]\), since the tail forcing is \(\le\kappa \)-distributive by \hlink{prepmainforcingtail}, we can intersect all of these dense open sets \(\bigcap _{s\in [\kappa ]^{<\omega }}d(s)_{G_0}\) and get another dense open set that intersects \(G_1\): there is a \(p\in G_1\cap d(s)_{G_0}\) for every \(s\in [\kappa ]^{<\omega }\).  Thus some \(q\in G_0\) has in \(\ffrak{M}\) that \(j(q)\Vdash \) \dblq{\(j(p)\in j(d)(s)\) for every \(s\in [j(\kappa )]^{<\omega }\)} and in particular \(j(q)\Vdash \dblq{j(p)\in \dot D}\).  Since \(q\in G_0\subseteq \bb{P}_\kappa \), \(j(q)\) is just \(q\) with a bunch of \(\dot{\bbone}\)s appended, meaning \(j(q)\in G*H_0\) and so indeed \(j(p)\in j"G_1\cap D\subseteq H_1\cap D\) in \(\ffrak{M}[G*H_0]\).\qedhere
		\end{proof}
		
		Thus \(j"G\subseteq G*H_0*H_1\) and so \(j:\mathrm{V}\rightarrow \mathrm{M}\) lifts to \(j^+:\mathrm{V}[G]\rightarrow \mathrm{M}[G*H]\).  It's not hard to see that \(j^+\) is still \(\lambda \)-strong since \(G\subseteq \mathrm{V}\mid\lambda =\mathrm{M}\mid\lambda \) and \(H\) adds no sets of rank \(<\lambda \).\qedhere
	\end{proof}
	
	More generally what this shows is the following.
	\begin{corollary}\label{liftembedding}
		If \(j:\mathrm{V}\rightarrow \mathrm{M}\) is \(\lambda \)-srs for \(\lambda \) a limit of strongs such that \(\lambda \) isn't stronger than a measurable in \(\ffrak{M}\), then we can lift \(j\) to \(j^+:\mathrm{V}^{\bbb{P}_\lambda }\rightarrow \mathrm{M}^{j(\bbb{P}_\lambda )}\).
	\end{corollary}
	And this gives the desired result.
	
	\begin{result}\label{concludingmaintheorem}
		Assume there is a proper class of srs cardinals and \(\axm{GCH}\) holds.  Therefore \(\bbb{P}\Vdash \axm{UWISS}+\dblq{\text{there is a proper class of strongs}}\)
	\end{result}
	\begin{proof}
		That \(\bbb{P}\Vdash \axm{UWISS}\) follows from \hlink{PforcesUWISS}.  For a proper class of strongs, let \(\kappa \) be srs and let \(\lambda >\kappa \) be strong in \(\ffrak{V}\).  By \hlink{mainforcingmainlem}, in \(\ffrak{V}^{\bbb{P}_\lambda }\), \(\kappa \) is \(\lambda \)-strong and this degree of strength is weakly indestructible, meaning \(\kappa \) is still \(\lambda \)-strong in all later stages (since, again, the tail forcings will be \(\le\kappa \)-distributive and \(<\kappa \)-strategi\-cally closed) and so \(\lambda \)-strong in \(\ffrak{V}^{\bbb{P}}\).  Since there are a proper class of \(\ffrak{V}\)-strongs above \(\kappa \), it follows that \(\kappa \) is strong in \(\ffrak{V}^{\bbb{P}}\).  Hence any \(\kappa \) srs in \(\ffrak{V}\) is strong in \(\ffrak{V}^{\bbb{P}}\), and we have a proper class of both.\qedhere
	\end{proof}
	
	This completes the proof of \hlink{mainforcing}, repeated below, and so in conjunction with \hlink{maincore}, \hlink{maingoal} holds.
	
	\begin{corollary}\label{mainforcing2}
		(1) implies (2) and (3) where these stand for
		\begin{enumerate}
			\item \(\Con(\axm{ZFC}\mathrel{+}\) \dblq{There is a proper class of srs cardinals}\()\).
			\item \(\Con(\axm{ZFC}\mathrel{+}\) \dblq{There is a proper class of strongs} \(\mathrel{+}\axm{UWISS})\).
			\item \(\Con(\axm{ZFC}\mathrel{+}\) \dblq{There is a proper class of strongs} \(+\) \dblq{Every \(\lambda \)-strong \(\kappa \) has weakly indestructible \(\lambda \)-strength whenever \(\lambda \) is below the least measurable limit of measurable cardinals larger than \(\kappa \)}.
		\end{enumerate}
	\end{corollary}
	Again, we can go further than even what (3) states; beyond the next measurable limit of measurable limits of measurables, and so on.  But at some point, saying this just becomes silly.  The point is that we get much more than weak indestructibility for \(\kappa +2\)-strength, up to the next cardinal \(\lambda \) that has at least as many measurables below it as a \(\lambda +2\)-strong cardinal should (since it \emph{was} \(\lambda +2\)-strong in \(\ffrak{V}^{\bbb{P}_{\kappa +1}}\)).
	
	The above results allow us to also prove \hlink{cohenresurrectdegrees}, restated below.
	
	\begin{result}\label{cohenresurrectdegreeslater}
		It's consistent (relative to two srs cardinals and proper class of strong cardinals) that we can force a \(\kappa \) that is not \(\rho \)-strong to be \(\rho \)-strong by a collapse \(\ffrak{Col}(\kappa ^+,\rho )\) 
	\end{result}
	\begin{proof}
		Consider the Easton support iteration \(\bbb{P}_\kappa =\bigast_{\alpha <\kappa }\dot{\bbb{Q}}_\alpha \) where \(\dot{\bbb{Q}}_\kappa \) is defined by the following.
		\begin{enumerate}
			\item At every stage \(\kappa \) that is (forced to be) \(\kappa +2\)-strong, we force with the lottery of \(\kappa ,\rho \)-appro\-priate posets for minimal \(\rho \) (if any exist) of rank \(<\kappa ^{+\P}\).
			\item Then with \(\ffrak{Col}(\kappa ^+,\rho )\).
			\item If there is no such \(\rho \), and \(\kappa \) was originally \(<\lambda \)-strong in \(\ffrak{V}\) where \(\kappa ^{+\sharp\sharp}\le\lambda <\kappa ^{+\P}\), then force with \(\ffrak{Col}(\kappa ^+,\lambda )\).
			\item Otherwise force trivially.
		\end{enumerate}
		If this collapses never resurrect degrees of strength, then the proof below goes through to show that the resulting srs cardinals are strong.  Moreover, weak indestructibility will hold for \emph{all} degrees of strength, simply because the tail will be appropriate (by \hlink{prepmainforcingtail}) and
		\begin{itemize}
			\item Anytime we destroy \(\kappa \)'s \(\rho \)-strength, we make sure the situation in \hyperlink{interactresurrect}{Figure 2} does not happen and \hyperlink{nointeractresurrect}{Figure 3} does: the next stage of forcing occurs far beyond \(\rho \), so the lack of \(\rho \)-strong extenders continues through to the end of the preparation.
			\item If we \emph{don't} destroy \(\kappa \)'s strength with a lottery at stage \(\kappa \), then \(\kappa \)'s strength in \(\ffrak{V}^{\bbb{P}_\kappa }\) indestructible at that stage.  The collapse, if there is one, in (3) doesn't add any degrees of strength by hypothesis.  Thus \(\kappa \)'s strength would be weakly indestructible in \(\ffrak{V}^{\bbb{P}_{\kappa +1}}\) (if there were some \(\kappa ,\rho \)-appro\-priate \(\dot{\bbb{Q}}\) then in \(\ffrak{V}^{\bbb{P}_\kappa }\), we could force with the \(\kappa ,\rho \)-appro\-priate \(\dot{\bbb{Q}}_\kappa *\dot{\bbb{Q}}\)).  The collapse will again ensure the next stage of forcing occurs far beyond \(\lambda \) as in \hyperlink{nointeractresurrect}{Figure 3}, so the lack of \(\lambda \)-strong extenders continues through to the end of the preparation.
		\end{itemize} 
		The same proof as \hlink{liftembedding} with the previous preparation in \hlink{theprepdef} tells us that the srs cardinals become strong in \(\ffrak{V}^{\bbb{P}}\).  But two strong cardinals and weak indestructibility for all degrees of strength contradicts \hlink{hamkinsoopsy}.\qedhere
	\end{proof}
	
	\section{Small Side Results} \label{Small Side Results}
	
	It's not hard to see that the forcing preparation \(\bbb{P}\) from \hlink{theprepdef} can be generalized and restricted to the following theorem when combined with the original proof from \cite{aptersargsyan}.  The idea really is that \cite{aptersargsyan} only gives \(\axm{UWISS}\) and the only reason there is a weakly indestructible strong cardinal in the end is that all of its degrees of strength are small: there is no measurable cardinal above the strong.
	
	\begin{theorem}\label{srshypergenforcing}
		Let \(\alpha \in \mathrm{Ord}\).  Therefore the following are equiconsistent.
		\begin{enumerate}
			\item \(\axm{ZFC}+\dblq{\text{there are (at least) }\alpha \text{ srs cardinals with a hyperstrong above them}}\);
			\item \(\axm{ZFC}+\axm{UWISS}+\dblq{\text{there are (at least) }\alpha +1\text{ strong cardinals}}\).  In fact, the last of those \(\alpha +1\)-strong cardinals is weakly indestructible for all of its degrees of strength.
		\end{enumerate}
	\end{theorem}
	\begin{proof}
		That (2) is relatively consistent relative to (1) follows from the techniques of \cite{aptersargsyan}: forcing \(\axm{GCH}\), and cutting off the universe at \(\kappa ^{+\sharp}\) whenever \(\kappa \) has weakly indestructible \(\kappa ^{+\sharp}\)-strength in \(\ffrak{V}^{\bbb{P}_\kappa }\).  This always ends at or before the final hyperstrong.  Such a \(\kappa \) is strong in the final model.  The lower srs cardinals become strong after forcing with \(\bbb{P}_\kappa \) as noted with \hlink{mainforcingmainlem}, and so there are \(\alpha +1\)-strong cardinals after forcing with \(\bbb{P}_\kappa \) and clearly \(\bbb{P}_\kappa \cong\bbb{P}_\mu \) forces \(\axm{UWISS}\) over \(\ffrak{V}\mid\mu \).
		
		That (1) is relatively consistent to (2) follows again from the techniques of \cite{aptersargsyan}: if \(\ffrak{V}\vDash \axm{ZFC}+\axm{UWISS}\), any \(\ffrak{V}\)-strong cardinal \(\kappa \) will be hyperstrong in \(\ffrak{K}\) and if there's a \(\ffrak{V}\)-strong above \(\kappa \), then \(\kappa \) will be srs in \(\ffrak{K}\).\qedhere
	\end{proof}
	
	Reflecting strongs is just one kind of reflection, but in principle, we could also enforce that we reflect more.  This essentially converges onto the idea of a Woodin cardinal.  The following results are proven in my upcoming thesis \cite{thesis}.
	
	\begin{theorem}\label{goal2}
		Let \(\bbb{P}\) be as in \hlink{theprepdef}. Let \(\delta \) be a Woodin cardinal.  Therefore \(\bbb{P}_\delta \Vdash \dblq{\delta \text{ is Woodin}}\).  Hence the existence of a Woodin cardinal \(\delta \) gives the consistency of a Woodin cardinal \(\delta \) where \(\ffrak{V}\mid\delta \vDash \axm{UWISS}\).
	\end{theorem}
	
	This gives the consistency of \(\axm{UWISS}\) with a proper class of a large variety of large cardinal notions relative to the existence of a Woodin cardinal.
	
	While on the topic of Woodins, consider the interaction of universal weak indestructibility for \emph{large} degrees of strength with Woodin cardinals.  For example, we used a cardinal that is strong and reflects strongs to get some weak indestructibility results for strength. We cannot have the same indestructibility for reflection properties.
	
	\begin{definition}\label{woodinforweakindestructibility}
		An inaccessible cardinal \(\delta \) is \emph{Woodin by weak indestructibility} iff for every \(A\subseteq \mathrm{V}\mid\delta \), there is a \(\kappa <\delta \) that is \(<\delta \)-strong reflecting \(A\) such that this strength and reflection is weakly indestructible by \(<\kappa \)-strategi\-cally closed posets.
	\end{definition}
	\begin{result}\label{nowoodinforwi}
		It is not possible to have a cardinal that is Woodin for weak indestructibility.
	\end{result}
	\begin{proof}
		Suppose not: let \(\delta \) be Woodin for weak indestructibility.  Let \(W\) be the set of all cardinals with weakly indestructible \(<\delta \)-strength in \(\ffrak{V}\).  There is therefore an (arbitrarily large) \(\kappa <\delta \) that is \(<\delta \)-strong reflecting \(W\) such that this is weakly indestructible.  In particular, after adding a Cohen subset \(A\subseteq \ffrak{Add}(\kappa ^+,1)\), \(\kappa \) is still \(<\delta \)-strong reflecting \(W\). Let \(j:\mathrm{V}[A]\rightarrow \mathrm{M}[j(A)]\), \(\cp(j)=\kappa \) witness this for some large \(\lambda <\delta \).  In \(\ffrak{M}[j(A)]\), \(j(W)\cap \lambda =W\cap \lambda \ni\kappa \).  Consider the least element \(\mu \in W\) above \(\kappa \) that is \(<\delta \)-strong, and assume without loss of generality that \(\lambda \) is large enough that \(\mu <\lambda \).  By \cite{hamkinssuper} in \(\ffrak{V}[A]\), since \(\ffrak{Add}(\kappa ^+,1)\) is small relative to \(\mu \), \(\mu \)'s strength is weakly destructible by \(\ffrak{Add}(\mu ^+,1)\).  Reflected in \(\ffrak{V}[A]\), this means arbitrarily large cardinals \(\hat\mu \in W\) have their \(<\delta \)-strength destroyed by \(\ffrak{Add}(\hat\mu ^+,1)\).  Hence \(\ffrak{Add}(\kappa ^+,1)*\ffrak{Add}(\hat\mu ^+,1)\) destroys \(\hat\mu \)'s \(<\delta \)-strength in \(\ffrak{V}\), contradicting that \(\hat\mu \in W\).\qedhere
	\end{proof}
	
	Despite this result, we can still have a Woodin cardinal such that the cardinals witnessing this can always be chosen to have weakly indestructible strength.
	
	\begin{definition}\label{Woodinwitnesswi}
		A cardinal \(\delta \) is \emph{Woodin witnessed by (weak) indestructibility} iff \(\delta \) is Woodin, and every \(<\delta \)-strong cardinal \(\kappa <\delta \) has (weakly) indestructible strength.
	\end{definition}
	Using strongs reflecting strongs again, it's possible to force any Woodin cardinal \(\delta \) to be witnessed by weak indestructibility. Stated differently, we get an equiconsistency.
	
	\begin{corollary}\label{woodinwitnessequi}
		The following are equiconsistent with \(\axm{ZFC}\).
		\begin{enumerate}
			\item There is a Woodin cardinal.
			\item There is a Woodin witnessed by weak indestructibility.
			\item There is a Woodin cardinal \(\delta \) such that \(\ffrak{V}\mid\delta \vDash \axm{UWISS}\).
		\end{enumerate}
	\end{corollary}
	
	This suggests there is a delicate balance between (weak) indestructibility and reflection: one can have weakly indestructible strength with relative ease, but too much of this precludes too much weakly indestructible reflection properties by \hlink{nowoodinforwi}.  What kinds of non-trivial reflection can be indestructible is a subject of further research.  For example, is it possible to have a proper class of strongs, and a strong reflecting strongs with this strength and reflection weakly indestructible?  Below a Woodin witnessed by weak indestructibility, \(\delta \), the answer is no, because in \(\ffrak{V}\mid\delta \), the weakly indestructible strong cardinals are just all of the strong cardinals. But in a more general setting the answer is not as obvious to me.

	Some of the open problems stated above are collected here for convenience.
	\begin{questions}\label{endproblems}
		All of the questions below can also be rephrased in terms of supercompactness.
		\begin{enumerate}
			\item Is it possible to have every \(\kappa \)'s \(<\kappa ^{+\sharp\sharp}\)-strength weakly indestructible in the presence of multiple strong cardinals?
			\item To what extent can the reflection properties in the embeddings of a measurable cardinal be (weakly) indestructible?
			\item Is it possible to have a strong reflecting strongs cardinal (with a strong above it) such that this strength and reflection of (ground model) strongs is weakly indestructible?
			\item To what extent can we control the resurrection of degrees after destroying degrees of strength in a preparation like \hlink{theprepdef}?
			\item If a poset is \(\le\kappa \)-strategi\-cally closed, is it \(<\kappa ^+\)-strategi\-cally closed?
		\end{enumerate}
	\end{questions}

	\newpage

		
	

\begin{bibdiv}
			\begin{biblist}
				\bib{aptersargsyan}{article}{
					title={An Equiconsistency for Universal Indestructibility},
					author={Apter, Arthur W.},
					author={Sargsyan, Grigor},
					date={2010},
					journal={The Journal of Symbolic Logic},
					volume={75},
					number={1},
					pages={314\textendash 322}
				}
				\bib{apterhamkins}{article}{
					title={Universal indestructibility},
					author={Apter, Arthur W.},
					author={Hamkins, Joel David},
					date={1999},
					journal={Kobe Journal of Mathematics},
					volume={16},
					number={2},
					pages={119\textendash 130}
				}
				\bib{apterapplications}{article}{
					title={Some applications of Sargsyan's equiconsistency method},
					author={Apter, Arthur W.},
					date={2012},
					journal={Fundamenta Mathematicae},
					volume={216},
					number={3},
					pages={207\textendash 222}
				}
				\bib{cummings}{article}{
					title={Iterated Forcing and Elementary Embeddings},
					author={Cummings, James},
					book={
						title={Handbook of Set Theory},
						editor={Foreman, Matthew},
						editor={Kanamori, Akihiro},
						publisher={Springer, Dordrecht},
						date={2010},
						volume={2}
					},
					pages={775\textendash 883}
				}
				\bib{GapForcing}{article}{
					title={Gap Forcing},
					author={Hamkins, Joel David},
					date={2001},
					journal={Israel Journal of Mathematics},
					volume={125},
					pages={237–252}
				}
				\bib{hamkinssuper}{article}{
					title={Small Forcing Makes Any Cardinal Superdestructible},
					author={Hamkins, Joel David},
					date={1998},
					journal={The Journal of Symbolic Logic},
					volume={63},
					number={1},
					pages={51–58},
					publisher={Association for Symbolic Logic}
				}
				\bib{thesis}{thesis}{
					title={Weak Indestructibility and Reflection},
					author={Holland, James},
					date={2022},
					type={Ph.D. Thesis},
					note={Unpublished}
				}
				\bib{jech}{book}{
					author={Jech, Thomas},
					title={Set Theory},
					subtitle={The Third Millennium Edition, revised and expanded},
					date={2003},
					series={Springer Monographs in Mathematics},
					publisher={Springer-Verlag Berlin Heidelberg}
				}
				\bib{Kanamori}{book}{
					author={Kanamori, Akihiro},
					title={The Higher Infinite},
					date={2009},
					edition={2},
					series={Springer Monographs in Mathematics},
					publisher={Springer-Verlag Berlin Heidelberg}
				}
				\bib{laverprep}{article}{
					title={Making the supercompactness of \(\kappa \) indestructible under \(\kappa \)-directed closed forcing},
					author={Laver, Richard},
					date={1978},
					journal={Israel Journal of Mathematics},
					volume={29},
					number={4},
					pages={385–388}
				}
				\bib{Steelfine}{book}{
					title={Fine Structure and Iteration Trees},
					author={Mitchell, William J.},
					author={Steel, John R.},
					date={2016},
					series={Lecture Notes in Logic},
					volume={3},
					publisher={Cambridge University Press},
				}
				\bib{Perlmutter}{article}{
					title={The large cardinals between supercompact and almost-huge},
					author={Perlmutter, Norman Lewis},
					date={2015},
					journal={Archive for Mathematical Logic},
					volume={54},
					pages={257–289}
				}
				\bib{Schindler}{article}{
					title={Iterates of the Core Model},
					author={Schindler, Ralf},
					date={2006},
					journal={The Journal of Symbolic Logic},
					volume={71},
					number={1},
					pages={241–251}
				}
				\bib{Steelcore}{book}{
					title={The Core Model Iterability Problem},
					author={Steel, John R.},
					date={2016},
					series={Lecture Notes in Logic},
					volume={8},
					publisher={Cambridge University Press},
				}
			\end{biblist}
		\end{bibdiv}
\end{document}